\documentclass[a4paper,12pt,final]{amsart}
\usepackage{times,a4wide,mathrsfs,amssymb,amsmath,amsthm,wasysym}

\newcommand{\C}{\mathbb{C}}

\newcommand{\QQ}{\mathbb{Q}}
\newcommand{\NN}{\mathbb{N}}
\newcommand{\PP}{\mathbb{P}}
\newcommand{\LL}{\mathcal{L}}

\newcommand{\OO}{\mathcal O}
\newcommand{\Ss}{\mathcal S}

\newcommand{\XX}{\mathcal X}
\newcommand{\YY}{\mathcal Y}

\newcommand{\VV}{\mathcal V}

\newcommand{\FF}{\mathcal F}

\newcommand{\Z}{\mathcal Z}

\newcommand{\MM}{\mathcal M}

\newcommand{\BB}{\mathfrak B}

\newcommand{\gr}{\hbox{Gr}}
\newcommand{\wt}{\widetilde}
\newcommand{\rom}{\romannumeral}

 \makeatletter
\newcommand*{\da@rightarrow}{\mathchar"0\hexnumber@\symAMSa 4B }
\newcommand*{\da@leftarrow}{\mathchar"0\hexnumber@\symAMSa 4C }
\newcommand*{\xdashrightarrow}[2][]{%
  \mathrel{%
    \mathpalette{\da@xarrow{#1}{#2}{}\da@rightarrow{\,}{}}{}%
  }%
}
\newcommand{\xdashleftarrow}[2][]{%
  \mathrel{%
    \mathpalette{\da@xarrow{#1}{#2}\da@leftarrow{}{}{\,}}{}%
  }%
}
\newcommand*{\da@xarrow}[7]{%
  \sbox0{$\ifx#7\scriptstyle\scriptscriptstyle\else\scriptstyle\fi#5#1#6\m@th$}%
  \sbox2{$\ifx#7\scriptstyle\scriptscriptstyle\else\scriptstyle\fi#5#2#6\m@th$}%
  \sbox4{$#7\dabar@\m@th$}%
  \dimen@=\wd0 %
  \ifdim\wd2 >\dimen@
    \dimen@=\wd2 %   
  \fi
  \count@=2 %
  \def\da@bars{\dabar@\dabar@}%
  \@whiledim\count@\wd4<\dimen@\do{%
    \advance\count@\@ne
    \expandafter\def\expandafter\da@bars\expandafter{%
      \da@bars
      \dabar@ 
    }%
  }%  
  \mathrel{#3}%
  \mathrel{%   
    \mathop{\da@bars}\limits
    \ifx\\#1\\%
    \else
      _{\copy0}%
    \fi
    \ifx\\#2\\%
    \else
      ^{\copy2}%
    \fi
  }%   
  \mathrel{#4}%
}
\makeatother

\DeclareMathOperator{\aut}{Aut}
\DeclareMathOperator{\bir}{Bir}
\DeclareMathOperator{\ide}{id}

\DeclareMathOperator{\ima}{Im}

\DeclareMathOperator{\og}{\mathbb{OG}}

\newtheorem{theorem}{Theorem}[section]

\newtheorem{lemma}[theorem]{Lemma}
\newtheorem{sublemma}[theorem]{Sublemma}
\newtheorem{corollary}[theorem]{Corollary}
\newtheorem{proposition}[theorem]{Proposition}
\newtheorem{conjecture}[theorem]{Conjecture}
\newtheorem{remark}[theorem]{Remark}
\newtheorem{definition}[theorem]{Definition}
\newtheorem{convention}{Conventions}
\newtheorem{question}[theorem]{Question}
\newtheorem{notation}[theorem]{Notation}

\newtheorem{nonumbering}{Theorem}

\newtheorem{nonumberingc}{Corollary}

\newtheorem{nonumberingt}{Acknowledgements}

\begin{document}
\author[Robert Laterveer]
{Robert Laterveer}

\address{Institut de Recherche Math\'ematique Avanc\'ee,
CNRS -- Universit\'e 
de Strasbourg,\
7 Rue Ren\'e Des\-car\-tes, 67084 Strasbourg CEDEX,
FRANCE.}
\email{robert.laterveer@math.unistra.fr}

\title{On the Chow groups of certain EPW sextics}

\begin{abstract} This note is about the Hilbert square $X=S^{[2]}$, where $S$ is a general $K3$ surface of degree $10$, and the anti--symplectic birational involution $\iota$ of $X$ constructed by O'Grady. The main result is that the action of $\iota$ on certain pieces of the Chow groups of $X$ is as expected by Bloch's conjecture. Since $X$ is birational to a double EPW sextic $X^\prime$, this has consequences for the Chow ring of the EPW sextic $Y\subset\PP^5$ associated to $X^\prime$.
 \end{abstract}

\keywords{Algebraic cycles, Chow groups, motives, hyperk\"ahler varieties, anti--symplectic involution, K3 surfaces, (double) EPW sextics, Beauville's splitting principle, multiplicative Chow--K\"unneth decomposition, spread of algebraic cycles}
\subjclass[2010]{Primary 14C15, 14C25, 14C30.}

\maketitle

\section{Introduction}

For a smooth projective variety $X$ over $\C$, let $A^i(X):=CH^i(X)_{\QQ}$ denote the Chow groups (i.e. the groups of codimension $i$ algebraic cycles on $X$ with $\QQ$--coefficients, modulo rational equivalence). 
Let $A^i_{hom}(X)$ and $A^i_{AJ}(X)\subset A^i(X)$ denote the subgroup of homologically trivial (resp. Abel--Jacobi trivial) cycles. 
It seems fair to say that Chow groups of codimension $i>1$ cycles are still poorly understood. To cite but one example, there is Bloch's famous conjecture (which famously is still open for surfaces of general type with geometric genus $0$):

\begin{conjecture}[Bloch \cite{B}]\label{blochconj0} Let $S$ be a smooth projective surface. Let $\Gamma\in A^2(S\times S)$ be a correspondence such that
  \[ \Gamma^\ast=0\colon\ \ \ H^2(X,\OO_X)\ \to\ H^2(X,\OO_X)\ .\]
  Then
  \[ \Gamma^\ast=0\colon\ \ \ A^2_{AJ}(X)\ \to\ A^2_{AJ}(X)\ .\]
  \end{conjecture}

For varieties of higher dimension, versions of conjecture \ref{blochconj0} can be stated for $0$--cycles and for codimension $2$--cycles:

\begin{conjecture}\label{blochconj} Let $X$ be a smooth projective variety of dimension $n$. Let $\Gamma\in A^n(X\times X)$ be a correspondence such that
  \[ \Gamma^\ast=0\colon\ \ \ H^n(X,\OO_X)\ \to\ H^n(X,\OO_X)\ .\]
  Then
  \[ \Gamma^\ast=0\colon\ \ \ F^n A^n_{}(X)\ \to\ A^n_{}(X)\ .\]
  \end{conjecture}
  
Here $F^\ast A^\ast()$ denotes the conjectural Bloch--Beilinson filtration, expected to exist for all smooth projective varieties \cite{J2}, \cite{J4}, \cite{Mur}, \cite{MNP}.
   
 \begin{conjecture}[]\label{blochconj2} Let $X$ be a smooth projective variety of dimension $n$. Let $\Gamma\in A^n(X\times X)$ be a correspondence such that
  \[ \Gamma^\ast=0\colon\ \ \ H^2(X,\OO_X)\ \to\ H^2(X,\OO_X)\ .\]
  Then
  \[ \Gamma^\ast=0\colon\ \ \ A^2_{AJ}(X)\ \to\ A^2_{AJ}(X)\ .\]
  \end{conjecture}  

Let us now restrict attention to hyperk\"ahler varieties $X$ (by which we mean: projective irreducible holomorphic symplectic manifolds \cite{Beau1}, \cite{Beau0}). For the purposes of this introduction, we will optimistically assume the Chow ring of $X$ has a splitting
    \[ A^i(X) = \bigoplus_j A^i_{(j)}(X)\ \]
  such that $A^\ast_{(\ast)}(X)$ is a bigraded ring, and the piece $A^i_{(j)}(X)$ is isomorphic to the graded
  $\gr^j_F A^i(X)$ for the conjectural Bloch--Beilinson filtration mentioned above. (This is expected to be the case for all hyperk\"ahler varieties \cite{Beau3}.)  
  
  Here is what conjectures \ref{blochconj} and \ref{blochconj2} predict for the action of an anti--symplectic involution (i.e., an involution acting as $-1$ on the symplectic form) on the Chow groups of $X$:

 \begin{conjecture}\label{blochHK} Let $X$ be a hyperk\"ahler variety of dimension $4$. Let $\iota$ be an anti--symplectic involution of $X$. Then
   \[  \begin{split}  \iota^\ast =-\ide\colon&\ \ \ A^i_{(2)}(X)\ \to\ A^i_{(2)}(X)\ \ \ \hbox{for\ }i=2,4\ ;\\
                             \iota^\ast =\ide\colon&\ \ \ A^4_{(4)}(X)\ \to\ A^4_{(4)}(X)\ .\\
                          \end{split}\]
     \end{conjecture}

(The statement for $A^4_{(4)}(X)= F^4 A^4(X)$ is conjecture \ref{blochconj} applied to the graph of $\iota$. The statement for $A^2_{(2)}(X)=A^2_{AJ}(X)$ is conjecture \ref{blochconj2}
applied to the graph of $\iota$. The statement for $A^4_{(2)}(X)$ then follows from the expected ``hard Lefschetz'' isomorphism $A^2_{(2)}(X)\xrightarrow{\cong} A^4_{(2)}(X)$.)

The main result of this note establishes a weak form of conjecture \ref{blochHK} for a $19$--dimensional family of hyperk\"ahler fourfolds:

\begin{nonumbering}[=theorem \ref{main}]
Let $X$ be the Hilbert scheme $S^{[2]}$, where $S$ is a very general $K3$ surface of degree $d=10$. Let $\iota\in\bir(X)$ be the anti--symplectic involution constructed by O'Grady \cite{OG2}. Then
   \[    \begin{split}         \iota^\ast=\ide\colon&\ \ \ A^4_{(0)}(X)\ \to\ A^4_{(0)}(X)\ ;\\      
                                    \iota^\ast =-\ide\colon&\ \ \ A^4_{(2)}(X)\ \to\ A^4_{(2)}(X)\ ;\\
                             (\Pi_2^X)_\ast \iota^\ast=-\ide \colon& \ \ \ A^2_{(2)}(X)\ \to\ A^2_{(2)}(X)\ ;\\           
                                         (\Pi_4^X)_\ast      \iota^\ast =\ide\colon&\ \ \ A^4_{(4)}(X)\ \to\ A^4_{(4)}(X)\ .\\
                          \end{split}\]
  \end{nonumbering}
 
 The birational involution $\iota$ of \cite{OG2} is briefly explained in proposition \ref{og1} below. 
 The notation $A^\ast_{(\ast)}(X)$ in theorem \ref{main} refers to the bigraded ring structure constructed unconditionally for all Hilbert squares of $K3$ surfaces by Shen--Vial, using their version of the Fourier transform on the Chow ring \cite{SV} (cf. also section \ref{secmck} below). The $\Pi_j^X$ refer to the Chow--K\"unneth decomposition of \cite{SV}; by construction, these have the property that $(\Pi_j^X)_\ast A^i(X)=A^i_{(2i-j)}(X)$.

 It is known that a variety $X$ as in theorem \ref{main} has a birational model $X^\prime$ which is a hyperk\"ahler variety; $X^\prime$ is a so--called {\em double EPW sextic\/} \cite{OG}, \cite{OG2}, \cite{OG3}. The variety $X^\prime$ has a generically $2:1$ morphism to a slightly singular sextic hypersurface $Y\subset\PP^5$, called an {\em EPW sextic \/}\cite{EPW}, \cite{OG}. Theorem \ref{main} has interesting consequences for the Chow ring of this EPW sextic:
   
  \begin{nonumberingc}[=corollary \ref{epw}] Let $X$ be as in theorem \ref{main}, and let 
  $Y\subset\PP^5$ be the associated EPW sextic.
  For any $r\in\NN$, let
  \[ E^\ast(Y^r)\ \subset\ A^\ast(Y^r) \]
  be the subring generated by (pullbacks of) $A^1(Y)$ and $A^2(Y)$. The cycle class map 
  \[   E^k(Y^r)\ \to\ \gr^W_{2k} H^{2k}(Y^r) \]
  is injective for $k\ge 4r-1$.
  \end{nonumberingc}                        

(Here, the Chow ring $A^\ast(Y^r)$ is taken to mean the operational Chow cohomology of Fulton--MacPherson \cite{F}.)

In particular, taking $r=1$, we find that the subspaces
  \[   \begin{split} &\ima \Bigl(  A^2(Y)\otimes A^1(Y)\ \to\ A^3(Y)\Bigr)\ ,\\
                             &\ima \Bigl(  A^2(Y)\otimes A^2(Y)\ \to\ A^4(Y)\Bigr)\ \\
                             \end{split}\]
  are of dimension $1$ (corollary \ref{epw1}). This is analogous to known results for $0$--cycles on $K3$ surfaces \cite{BV} and on certain Calabi--Yau varieties \cite{V13}, \cite{LFu} (cf. remark \ref{compare} below).                      
                             
Theorem \ref{main} is proven using the technique of ``spread'' of algebraic cycles in a family, as developed by Voisin in her seminal work on the Bloch/Hodge equivalence for complete intersections \cite{V0}, \cite{V1}, \cite{V8}, \cite{Vo}. 

In a final section (section \ref{finsec}), some questions related to theorem \ref{main} are stated, which we hope may spurn further research.

 \vskip0.6cm

\begin{convention} In this article, the word {\sl variety\/} will refer to a reduced irreducible scheme of finite type over $\C$. A {\sl subvariety\/} is a (possibly reducible) reduced subscheme which is equidimensional. 

{\bf All Chow groups will be with rational coefficients}: we will denote by $A_j(X)$ the Chow group of $j$--dimensional cycles on $X$ with $\QQ$--coefficients; for $X$ smooth of dimension $n$ the notations $A_j(X)$ and $A^{n-j}(X)$ are used interchangeably. 

The notations $A^j_{hom}(X)$, $A^j_{AJ}(X)$ will be used to indicate the subgroups of homologically trivial, resp. Abel--Jacobi trivial cycles.
For a morphism $f\colon X\to Y$, we will write $\Gamma_f\in A_\ast(X\times Y)$ for the graph of $f$.
The contravariant category of Chow motives (i.e., pure motives with respect to rational equivalence as in \cite{Sc}, \cite{MNP}) will be denoted $\MM_{\rm rat}$.

If $\tau\colon Y\to X$ is an inclusion of smooth varieties and $b\in A^j(X)$, we will often write
  \[   b\vert_{Y}\ \ \in A^j(Y)\]
  to indicate the class $\tau^\ast(b)$.

%The Griffiths group $\grif^j$ is the group of codimension $j$ cycles that are homologically trivial modulo algebraic equivalence, again with $\QQ$--coefficients. 

We use $H^j(X)$ 
to indicate singular cohomology $H^j(X,\QQ)$.

We write $\aut(X)$ and $\bir(X)$ to denote the group of automorphisms, resp. of birational automorphisms, of $X$.

Given an involution $\iota\in\bir(X)$, we will write $A^j(X)^\iota$ (and $H^j(X)^\iota$) for the subgroup of $A^j(X)$ (resp. $H^j(X)$) invariant under $\iota$.
\end{convention}

\section{Preliminaries}

%\subsection{Quotient varieties}

%\begin{definition} A {\em projective quotient variety\/} is a variety
%  \[ X=Y/G\ ,\]
 % where $Y$ is a smooth projective variety and $G\subset\hbox{Aut}(Y)$ is a finite group.
%  \end{definition}
  
% \begin{proposition}[Fulton \cite{F}]\label{quot} Let $X$ be a projective quotient variety of dimension $n$. Let $A^\ast(X)$ denote the operational Chow cohomology ring. The natural map
 %  \[ A^i(X)\ \to\ A_{n-i}(X) \]
 %  is an isomorphism for all $i$.
  % \end{proposition}
   
  % \begin{proof} This is \cite[Example 17.4.10]{F}.
    %  \end{proof}

%\begin{remark} It follows from proposition \ref{quot} that the formalism of correspondences goes through unchanged for projective quotient varieties (this is also noted in \cite[Example 16.1.13]{F}). 
 % \end{remark}

\subsection{MCK decomposition}
\label{ssmck}

\begin{definition}[Murre \cite{Mur}] Let $X$ be a smooth projective variety of dimension $n$. We say that $X$ has a {\em CK decomposition\/} if there exists a decomposition of the diagonal
   \[ \Delta_X= \Pi_0^X+ \Pi_1^X+\cdots +\Pi_{2n}^X\ \ \ \hbox{in}\ A^n(X\times X)\ ,\]
  such that the $\Pi_i^X$ are mutually orthogonal idempotents and $(\Pi_i^X)_\ast H^\ast(X)= H^i(X)$.
  
  (NB: ``CK decomposition'' is shorthand for ``Chow--K\"unneth decomposition''.)
\end{definition}

\begin{remark} The existence of a CK decomposition for any smooth projective variety is part of Murre's conjectures \cite{Mur}, \cite{J2}. 
%If a quotient variety $X$
%has finite--dimensional motive, and the K\"unneth components are algebraic, then $X$ has a CK decomposition (this can be proven just as \cite{J2}, where this is stated for smooth $X$).
\end{remark}

\begin{definition}[Shen--Vial \cite{SV}] Let $X$ be a smooth projective variety of dimension $n$. Let $\Delta^X_{sm}\in A^{2n}(X\times X\times X)$ be the class of the small diagonal
  \[ \Delta^X_{sm}:=\bigl\{ (x,x,x)\ \vert\ x\in X\bigr\}\ \subset\ X\times X\times X\ .\]
  An MCK decomposition is a CK decomposition $\{\Pi_i^X\}$ of $X$ that is {\em multiplicative\/}, i.e. it satisfies
  \[ \Pi_k^X\circ \Delta^X_{sm}\circ (\Pi_i^X\times \Pi_j^X)=0\ \ \ \hbox{in}\ A^{2n}(X\times X\times X)\ \ \ \hbox{for\ all\ }i+j\not=k\ .\]
  
 (NB: ``MCK decomposition'' is shorthand for ``multiplicative Chow--K\"unneth decomposition''.) 
  \end{definition}
  
  \begin{remark} The small diagonal (seen as a correspondence from $X\times X$ to $X$) induces the {\em multiplication morphism\/}
    \[ \Delta^X_{sm}\colon\ \  h(X)\otimes h(X)\ \to\ h(X)\ \ \ \hbox{in}\ \MM_{\rm rat}\ .\]
 Suppose $X$ has a CK decomposition
  \[ h(X)=\bigoplus_{i=0}^{2n} h^i(X)\ \ \ \hbox{in}\ \MM_{\rm rat}\ .\]
  By definition, this decomposition is multiplicative if for any $i,j$ the composition
  \[ h^i(X)\otimes h^j(X)\ \to\ h(X)\otimes h(X)\ \xrightarrow{\Delta^X_{sm}}\ h(X)\ \ \ \hbox{in}\ \MM_{\rm rat}\]
  factors through $h^{i+j}(X)$.
  It follows that if $X$ has an MCK decomposition, then setting
    \[ A^i_{(j)}(X):= (\Pi^X_{2i-j})_\ast A^i(X) \ ,\]
    one obtains a bigraded ring structure on the Chow ring: that is, the intersection product sends 
    $A^i_{(j)}(X)\otimes A^{i^\prime}_{(j^\prime)}(X) $ to  $A^{i+i^\prime}_{(j+j^\prime)}(X)$.
      
  The property of having an MCK decomposition is severely restrictive, and is closely related to Beauville's ``weak splitting property'' \cite{Beau3}. For more ample discussion, and examples of varieties with an MCK decomposition, we refer to \cite[Section 8]{SV} and also \cite{V6}, \cite{SV2}, \cite{FTV}, \cite{LV}.
    \end{remark}

  \begin{lemma}[Vial \cite{V6}]\label{hk} Let $X, X^\prime$ be birational hyperk\"ahler varieties. Then $X$ has an MCK decomposition if and only if $X^\prime$ has one.
  \end{lemma}
  
  \begin{proof} This is noted in \cite[Introduction]{V6}; the idea (as indicated in loc. cit.) is that Rie\ss's result \cite{Rie} implies that $X$ and $X^\prime$ have isomorphic Chow motives and the isomorphism is compatible with the multiplicative structure. 

          \end{proof}       

\subsection{MCK for $K3^{[2]}$}
\label{secmck}

\begin{theorem}[Shen--Vial \cite{SV}]\label{mck} Let $S$ be a $K3$ surface, and $X=S^{[2]}$. There exists an MCK decomposition $\{ \Pi_j^X\}$ for $X$. In particular,
setting
  \[ A^i_{(j)}(X):= (\Pi^X_{2i-j})_\ast A^i(X)\ \]
  defines a bigraded ring structure $A^\ast_{(\ast)}(X)$ on $A^\ast(X)$. Moreover, $A^\ast_{(\ast)}(X)$ coincides with the bigrading on $A^\ast(X)$ defined by the Fourier transform.
\end{theorem}

\begin{proof} The existence of $\{ \Pi_j^X\}$ is a special case of \cite[Theorem 13.4]{SV}. The ``moreover'' part is \cite[Theorem 15.8]{SV}.
\end{proof}

\begin{remark} The first statement of theorem \ref{mck} actually holds for $X=S^{[r]}$ for any $r\in\NN$ \cite{V6}.
\end{remark}

Any $K3$ surface $S$ has an MCK decomposition \cite[Example 8.17]{SV}. Since this property is stable under products \cite[Theorem 8.6]{SV}, $S^2$ also has an MCK decomposition. The following lemma records a basic compatibility between the bigradings on $A^\ast(S^{[2]})$ and on $A^\ast(S^2)$:

\begin{lemma}\label{compat} Let $S$ be a $K3$ surface, and $X=S^{[2]}$. Let $\Psi\in A^4(X\times S^2)$ be the correspondence coming from the diagram
  \[ \begin{array}[c]{ccc}
          S^{[2]} & \xleftarrow{} &\wt{S^2}\\
         {\scriptstyle h} \downarrow\ \ \  && \downarrow\\
          S^{(2)} & \xleftarrow{g}& S^2\\
          \end{array}\]
       (the arrow labelled $h$ is the Hilbert--Chow morphism; the right vertical arrow is the blow--up of the diagonal). Then
     \[  \begin{split} &(\Psi)_\ast R(X)\ \subset\ R(S^2)\ ,\\     
                            &({}^t\Psi)_\ast R(S^2)\ \subset\ R(X)\ ,\\
                            \end{split}\]
                    where $R=A^4_{(4)}$ or $A^4_{(2)}$ or $A^2_{(2)}$ or $A^2_{(0)}\cap A^2_{hom}$.        
\end{lemma}

\begin{proof} We prove the statement for ${}^t \Psi$ and $ R= A^2_{(2)}$ or $A^2_{(0)}\cap A^2_{hom}$, which are the only cases we'll be using (the other statements can be proven similarly). By construction of the MCK decomposition for $X$, there is a relation
  \begin{equation}\label{XS}  \Pi_k^X= {1\over 2}\ {}^t \Psi\circ \Pi_k^{S^2}\circ \Psi + \hbox{Rest}\ \ \ \hbox{in}\ A^4(X\times X)\ , \ \ \ (k=0,2,4,6,8)\ ,\end{equation}
  where $\{\Pi_k^{S^2}\}$ is a product MCK decomposition for $S^2$, and  ``Rest'' is a term coming from $\Delta_S\subset S\times S$ which does not act on $A^4(X)$ and on $A^2_{AJ}(X)$.
Since ${1\over 2}\ {}^t \Psi\circ \Psi$ is the identity on $A^2_{hom}(X)=A^2_{AJ}(X)$, we can write
  \[ ({}^t \Psi)_\ast (\Pi_k^{S^2})_\ast = ({}^t \Psi \circ\Pi_k^{S^2})_\ast = ({1\over 2}\ {}^t \Psi \circ \Psi \circ {}^t \Psi\circ\Pi_k^{S^2})_\ast\colon\ \ \ A^2_{hom}(S^2)\ \to\ A^2_{hom}(X)\ .\]
  In view of sublemma \ref{sub} below, this implies
   \begin{equation}\label{thishere} ({}^t \Psi)_\ast (\Pi_k^{S^2})_\ast =  ({1\over 2}\ {}^t \Psi \circ \Pi_k^{S^2} \circ \Psi \circ {}^t \Psi)_\ast\colon\ \ \ A^2_{hom}(S^2)\ \to\ A^2_{hom}(X)\ .\end{equation}  
   But then, plugging in relation (\ref{XS}), we find 
   \[ ({}^t \Psi)_\ast (\Pi_k^{S^2})_\ast A^2_{hom}(S^2)\ \subset\ (\Pi_k^X)_\ast A^2_{hom}(X)\ .\]
  Taking $k=2$, this proves
    \[ ({}^t \Psi)_\ast  A^2_{(2)}(S^2)\ \subset\  A^2_{(2)}(X)\ .\]
    Taking $k=4$, this proves
    \[ ({}^t \Psi)_\ast \Bigl( A^2_{(0)}(S^2)\cap A^2_{hom}(S^2)\Bigr)\ \subset\ A^2_{(0)}(X)\cap A^2_{hom}(X)\ .\]
     
   \begin{sublemma}\label{sub} There is commutativity
     \[ \Psi\circ {}^t \Psi\circ \Pi_k^{S^2} =    \Pi_k^{S^2}\circ\Psi\circ {}^t \Psi\ \ \ \hbox{in}\ A^4(S^4)\ .\]
     \end{sublemma}
     
   To prove the sublemma, we remark that $h_\ast h^\ast=2\ide\colon A^i(S^{(2)})\to A^i(S^{(2)})$, and so
    \[ (\Psi\circ {}^t \Psi)_\ast = 2\ g^\ast g_\ast = 2 (\Delta_{S^2}+\Gamma_\tau)_\ast\colon\ \ \   A^i(S^{2})\to A^i(S^{2})\ ,\]
    where $\tau$ denotes the involution switching the two factors. But $\{\Pi_k^{S^2}\}$, being a product decomposition, is symmetric and hence
    \[ \Gamma_\tau\circ \Pi_k^{S^2} \circ \Gamma_\tau= (\tau\times\tau)^\ast \Pi_k^{S^2}= \Pi_k^{S^2} \ \ \ \hbox{in}\ A^4(S^4)\ .\]
    This implies commutativity
    \[ \Gamma_\tau\circ \Pi_k^{S^2} =  \Pi_k^{S^2}\circ \Gamma_\tau\ \ \ \hbox{in}\ A^4(S^4)\ ,\]
    which proves the sublemma.
     \end{proof}

 \begin{remark} Lemma \ref{compat} is probably true for any $(i,j)$ (i.e., $\Psi$ should be ``of pure grade $0$'' in the language of \cite[Definition 1.1]{SV2}). I have not been able to prove this.
  \end{remark}

\subsection{Relative MCK for $S^2$}

\begin{notation} Let $\Ss\to B$ be a family (i.e., a smooth projective morphism). For $r\in\NN$, we write $\Ss^{r/B}$ for the relative $r$--fold fibre product
  \[ \Ss^{r/B}:= \Ss\times_B \Ss\times_B \cdots \times_B \Ss \ \]
  ($r$ copies of $\Ss$).
  \end{notation}

\begin{proposition}\label{prod} Let $\Ss\to B$ be a family of $K3$ surfaces. There exist relative correspondences 
   \[  \Pi_j^{\Ss^{2/B}}\ \ \in A^4(\Ss^{4/B})\ \ \  (j=0,2,4,6,8)\ ,\]
  such that for each $b\in B$, the restriction
  \[ \Pi_j^{(S_b)^2} := \Pi_j^{\Ss^{2/B}}\vert_{(S_b)^4}\ \ \ \in A^4((S_b)^4)\]
  defines a self--dual MCK decomposition for $(S_b)^2$.
  \end{proposition}

\begin{proof}

 On any $K3$ surface $S_b$, there is the distinguished $0$--cycle ${\mathfrak o}_{S_b}$ such that $c_2(S_b)=24 {\mathfrak o}_{S_b}$ \cite{BV}. Let $p_i\colon \Ss\times_B \Ss\to \Ss$, $i=1,2$, denote the projections to the two factors. Let $T_{\Ss/B}$ denote the relative tangent bundle.
The assignment
  \[ \begin{split} \Pi_0^\Ss &:= (p_1)^\ast \bigl({1\over 24} c_2(T_{\Ss/B})\bigr) \ \ \ A^2(\Ss\times_B \Ss)\ ,\\
                         \Pi_4^\Ss &:= (p_2)^\ast \bigl({1\over 24} c_2(T_{\Ss/B})\bigr) \ \ \ A^2(\Ss\times_B \Ss)\ ,\\
                         \Pi_2^\Ss &:= \Delta_\Ss - \Pi_0^\Ss - \Pi_4^\Ss\\
                    \end{split}\]
          defines (by restriction) an MCK decomposition for each fibre:
          \[  \Pi_j^{S_b}:= \Pi_j^\Ss\vert_{S_b\times S_b}\ \ \ \in A^2(S_b\times S_b)\ \ \ (j=0,2,4) \]
          is an MCK decomposition for any $b\in B$ \cite[Example 8.17]{SV}.
          
  Next, we consider the fourfold relative fibre product $\Ss^{4/B}$. Let
    \[ p_{ij}\colon \Ss^{4/B}\ \to\ \Ss^{2/B} \ \ \ (1\le i<j\le 4)\]
    denote projection to the $i$-th and $j$-th factor. We set
    \[  \Pi_j^{\Ss^{2/B}} := {\displaystyle \sum_{k+\ell=j}}  (p_{13})^\ast ( \Pi_k^{\Ss})\cdot (p_{24})^\ast (\Pi_\ell^\Ss)\ \ \ \in A^4(\Ss^{4/B})\ ,\ \ \ (j=0,2,4,6,8)\ .\]
    By construction, the restriction to each fibre induces an MCK decomposition (the ``product MCK decomposition'')
    \[  \Pi_j^{(S_b)^2} :=  \Pi_j^{\Ss^{2/B}}\vert_{(S_b)^4} = {\displaystyle \sum_{k+\ell=j}}  \Pi_k^{S_b}\times \Pi_\ell^{S_b}\ \ \ \in A^4((S_b)^4)\ ,\ \ \ (j=0,2,4,6,8)\ .\]
    \end{proof}
    
    \begin{proposition}\label{prod2} Let $\Ss\to B$ be a family of $K3$ surfaces. There exist relative correspondences
    \[  \Theta_1, \Theta_2\in A^{4}((\Ss\times_B \Ss)\times_B \Ss)\ ,\ \ \ \Xi_1, \Xi_2\in  A^{2}(\Ss\times_B (\Ss\times_B \Ss))  \]
    such that for each $b\in B$, the composition
    \[  A^2_{}(S_b\times S_b)\ \xrightarrow{((\Theta_1\vert_{(S_b)^3})_\ast, (\Theta_2\vert_{(S_b)^3})_\ast)}\
             A^2(S_b)\oplus A^2(S_b)\ \xrightarrow{(\Xi_1\vert_{(S_b)^3})_\ast  + (\Xi_2\vert_{(S_b)^3})_\ast    }\ A^2(S_b\times S_b) \]
      acts as a projector on $A^2_{(2)}(S_b\times S_b)$.
    \end{proposition}
    
    \begin{proof} 
    %Let $h\in A^1(\Ss)$ be a relative ample divisor, and let  
   % \[L^2:=\Delta_\Ss\cdot (p_1)^\ast(h^2) \in A^4(\Ss\times_B \Ss)\] 
   % be the relative correspondence with the property that for each $b\in B$, the restriction
  %  \[  (L^2)\vert_{S_b\times S_b}\ \ \ \in A^4(S_b\times S_b) \]
   % is a correspondence acting as multiplication by $h^2\vert_{S_b}$. 
   As before, let 
    \[ p_{ij}\colon\ \ \  \Ss^{4/B}\ \to\ \Ss^{2/B} \ \ \ (1\le i<j\le 4)\]
    denote projection to the $i$-th and $j$-th factor, and let 
    \[ p_i\colon \ \ \  \Ss^{2/B}\ \to\ \Ss \ \ \  (i=1,2) \]
    denote projection to the $i$--th factor.
            
        By construction of $\Pi_2^{\Ss^{2/B}}$, for each $b\in B$ we have equality
       \begin{equation}\label{both} \begin{split} ( \Pi_2^{\Ss^{2/B}})\vert_{(S_b)^4} =  {1\over 24^2}\Bigl(  {}^t \Gamma_{p_{1}}\circ \Pi_2^\Ss\circ \Gamma_{p_{1}}\circ 
               \bigl(   (p_{13})^\ast (\Delta_\Ss )\cdot (p_{2})^\ast c_2(T_{\Ss/B}) \cdot  (p_{4})^\ast c_2(T_{\Ss/B})&\bigr)\\  
               +
                           {}^t \Gamma_{p_{2}}\circ \Pi_2^\Ss\circ \Gamma_{p_{2}}\circ 
               \bigl(   (p_{24})^\ast (\Delta_\Ss )\cdot (p_{1})^\ast c_2(T_{\Ss/B})  \cdot (p_{3})^\ast c_2(T_{\Ss/B})   \bigr)   \Bigr)&\vert_{(S_b)^4}\\
               \ \ \ \hbox{in}\ A^4((&S_b)^4)\ .\\
               \end{split}\end{equation}
          %     where $d:=\deg(h^2\vert_{S_b})$.
         Indeed, using Lieberman's lemma \cite[16.1.1]{F}, we find that
         \[ \begin{split} &( {}^t \Gamma_{p_{1}}\circ \Pi_2^\Ss\circ \Gamma_{p_{1}})\vert_{(S_b)^4} = \bigl(({}^t \Gamma_{p_{13}})_\ast 
         (\Pi_2^{\Ss})\bigr)\vert_{(S_b)^4} =
           \bigl((p_{13})^\ast (\Pi_2^{\Ss})\bigr)\vert_{(S_b)^4}\ ,\\
           &( {}^t \Gamma_{p_{2}}\circ \Pi_2^\Ss\circ \Gamma_{p_{2}})\vert_{(S_b)^4} = \bigl(({}^t \Gamma_{p_{24}})_\ast 
         (\Pi_2^{\Ss})\bigr)\vert_{(S_b)^4} =
           \bigl((p_{24})^\ast (\Pi_2^{\Ss})\bigr)\vert_{(S_b)^4}\ ,\\
           \end{split}           \]
            
               and so   both sides of (\ref{both}) are equal to
         \[ \Pi_2^{S_b}\times \Pi_0^{S_b} +      \Pi_0^{S_b}\times \Pi_2^{S_b}    \ \ \ \in A^2((S_b)^4)\ .\]
                    
    It follows that if we define
    \[ \begin{split}
           \Theta_1&:={1\over 24^2} \, \Gamma_{p_{1}}\circ 
               \bigl(   (p_{13})^\ast (\Delta_\Ss )\cdot  (p_{2})^\ast c_2(T_{\Ss/B}) \cdot  (p_{4})^\ast c_2(T_{\Ss/B})   \bigr)\ \ \ \in A^4((\Ss\times_B \Ss)\times_B \Ss)\ ,\\  
           \Theta_2&:=  {1\over 24^2}\,  \Gamma_{p_{2}}\circ 
                 \bigl(   (p_{24})^\ast (\Delta_\Ss )\cdot (p_{1})^\ast c_2(T_{\Ss/B})  \cdot (p_{3})^\ast c_2(T_{\Ss/B}) \bigr)  \ \ \ \in A^4((\Ss\times_B \Ss)\times_B \Ss)\ ,\\  
            \Xi_1&:= {}^t \Gamma_{p_{1}}\circ \Pi_2^\Ss\ \ \ \in A^2(\Ss\times_B (\Ss\times_B \Ss)) \ ,\\
            \Xi_2&:=   {}^t \Gamma_{p_{2}}\circ \Pi_2^\Ss\ \ \ \in A^2(\Ss\times_B (\Ss\times_B \Ss)) \ ,\\
            \end{split}\]
      then we have
      \[ \Bigl( (\Xi_1\circ \Theta_1 + \Xi_2\circ \Theta_2)\vert_{(S_b)^4}   \Bigr){}_\ast =(\Pi_2^{(S_b)^2})_\ast   \colon\
      \ \ A^2_{}(S_b\times S_b)\ \to\ A^2_{(2)}(S_b\times S_b)\ \ \ \forall b\in B\ .\]      
      This proves the proposition.      
         \end{proof}

\subsection{Relative MCK for $S^{[2]}$}

\begin{proposition}\label{relmck} Let $\Ss\to B$ be a family of $K3$ surfaces (i.e. each fibre $S_b$ is a $K3$ surface), and let $\XX\to B$ be the family of associated Hilbert schemes (i.e., a fibre
$X_b$ is $(S_b)^{[2]}$). There exist relative correspondences
  \[ \Pi_j^{\XX}\ \ \in A^4(\XX\times_B \XX)\ \ \ (j=0,2,4,6,8)\ ,\]
  such that for each $b\in B$, the restrictions
   \[ \Pi_j^{X_b}:=  \Pi_j^{\XX}\vert_{X_b\times X_b}   \ \ \in A^4(X_b\times X_b)\ \ \ (j=0,2,4,6,8)\ \]
   define an MCK decomposition for $X_b$.
   \end{proposition}
   
 \begin{proof} The construction of an MCK decomposition for $X_b$ given in \cite[Theorem 13.4]{SV} can be done in a relative setting. That is, let $\{\Pi_j^{\Ss}\}$ be a relative 
 MCK decomposition for $\Ss$ as in proposition \ref{prod}, and let
 $\{ \Pi_j^{\Ss^{2/B}}\}$ be the induced relative MCK decomposition for $\Ss^{2/B}$ as in proposition \ref{prod}. Let
   \[ \Z\ \to\ B \]
   be the family obtained by blowing--up $\Ss\times_B \Ss$ along the relative diagonal $\Delta_\Ss$. As in the proof of \cite[Propositions 13.2 and 13.3]{SV}\footnote{The statement and proof of \cite[Proposition 13.2]{SV} should be slightly modified, as noted in \cite[Remark 2.8]{SV2}.}, one can use $\{ \Pi_j^{\Ss^{2/B}}\}$ and  $\{\Pi_j^{\Ss}\}$ to define relative correspondences
   \[  \Pi_j^\Z \ \ \ \in A^4(\Z\times_B \Z)\ \ \ (j=0,2,4,6,8)\ ,\]
   which restrict to an MCK decomposition of each fibre $Z_b$. Let
   \[ p\colon\ \ \Z\ \to\ \XX \]
   denote the morphism of $B$-schemes induced by the action of the symmetric group ${\mathfrak S}_2$, and let $\Gamma_p\in A^4(\Z\times_B \XX)$ be the graph of $p$.
   We define
   \[ \Pi_j^\XX:=  {1\over 2} \Gamma_p\circ \Pi_j^\Z\circ {}^t \Gamma_p\ \ \ \in A^4(\XX\times_B \XX)\ \ \ (j=0,2,4,6,8)\ .\]
   The restrictions 
   $\Pi_j^{X_b}:=  \Pi_j^{\XX}\vert_{X_b\times X_b}$ define an MCK decomposition for each fibre by \cite[Theorem 13.4]{SV}.   
     \end{proof}

\subsection{Multiplicative structure of Chow ring of $K3^{[2]}$}

\begin{theorem}[Shen--Vial \cite{SV}]\label{mult} Let $S$ be a $K3$ surface, and $X=S^{[2]}$. 

\noindent
{(\rom1)} Intersection product induces a surjection
  \[ A^2_{(2)}(X)\otimes A^2_{(2)}(X)\ \twoheadrightarrow\ A^4_{(4)}(X)\ .\]
  
  \noindent
  {(\rom2)} There is a distinguished class $l\in A^2_{(0)}(X)$ such that intersection induces an isomorphism
    \[ \cdot l\colon\ \ \ A^2_{(2)}(X)\ \xrightarrow{\cong}\ A^4_{(2)}(X)\ .\]

\end{theorem}    
    
 \begin{proof} This is \cite[Theorem 3]{SV}.
  \end{proof}

\subsection{Refined CK decomposition}

\begin{theorem}[Vial \cite{V4}]\label{pi_2} Let $X$ be a smooth projective variety of dimension $n\le 5$. Assume the Lefschetz standard conjecture $B(X)$ holds (in particular, the K\"unneth components $\pi_i\in H^{2n}(X\times X)$ are algebraic). Then there is a splitting into mutually orthogonal idempotents
  \[ \pi_i=\sum_j \pi_{i,j}\ \ \ \in H^{2n}(X\times X)\ ,\]
  such that
   \[ (\pi_{i,j})_\ast H^\ast(X) =gr^j_{\wt{N}} H^i(X)\ ,\]
   where $wt{N}^\ast$ is the {\em niveau filtration\/} of \cite{V4}.
   
In particular,  
   \[ \begin{split} &(\pi_{2,1})_\ast H^j(X) = H^{2}(X)\cap F^1\ ,\\
                          &(\pi_{2,0})_\ast H^j(X)= H^2_{tr}(X)\ .\\
                       \end{split}     \]
                      (Here $F^\ast$ denotes the Hodge filtration, and $H^2_{tr}(X)$ is the orthogonal complement to $H^2(X)\cap F^1$ under the pairing
        \[ \begin{split} H^2(X)\otimes H^2(X)\ &\to\ \QQ\ ,\\
                                a\otimes b\ &\mapsto\ a\cup h^{n-2}\cup b\ .)\\
                         \end{split} \]     
         The projector $\pi_{2,1}$ is supported on $C\times D$, where $C\subset X$ is a curve and $D\subset X$ is a divisor.                
                   \end{theorem}
 
 \begin{proof} This is \cite[Theorem 1]{V4}.
 \end{proof}

 \subsection{Mukai models}
 
 \begin{theorem}[Mukai \cite{Muk0}]\label{muk} Let $S$ be a general $K3$ surface of degree $10$ (i.e. genus $g(S)=6$). Let $G=G(2,5)$ denote the Grassmannian of lines in $\PP^4$. Then $S$ is isomorphic to the zero locus of a section of $\OO_G(1)^{\oplus 3}\oplus \OO_G(2)$.
 \end{theorem}
 
 This result can be exploited as follows:
 
  \begin{proposition}[Voisin \cite{V0}]\label{vanish} Let $\Ss\to B$ be the universal family of degree $10$ $K3$ surfaces (i.e., $B$ is a Zariski open in a product of projective spaces parametrizing sections of $\OO_G(1)^{\oplus 3}\oplus \OO_G(2)$ that are smooth). 
  We have
   \[ A^2_{hom}(\Ss\times_B \Ss)=0\ .\]
   \end{proposition}
   
   \begin{proof} The family $\Ss\to B$ is the family of smooth complete intersections $S_b\subset G$ defined by the very ample line bundles $\OO_G(1)$ ($3$ copies) and $\OO_G(2)$. The Grassmannian $G$ has trivial Chow groups. The result is thus a special case of \cite[Proposition 3.13]{V0} (NB: as explained in \cite[Section 3.3]{V0}, the hypothesis that \cite[Conjecture 1.6]{V0} holds is satisfied in codimension $2$, and so the result is unconditional in codimension $2$). 
   \end{proof}

 \subsection{EPW sextics}

 \begin{definition}[\cite{EPW}] Let $A\subset \wedge^3 \C^6$ be a subspace which is Lagrangian with respect to the symplectic form on $\wedge^3 \C^6$ given by the wedge product. The {\em EPW sextic associated to $A$\/} is
   \[ Y_A:= \Bigl\{  [v]\in \PP(\C^6)\ \vert\ \dim \bigl( A\cap ( v\wedge \wedge^2 \C^6)\bigr) \ge 1\Bigr\}\ \ \subset \PP(\C^6)\  .\]
 An {\em EPW sextic\/} is an $Y_A$ for some $A\subset \wedge^3 \C^6$ Lagrangian.  
 \end{definition}
 
 \begin{proposition}[O'Grady \cite{OG2}]\label{og1} Let $X=S^{[2]}$ where $S$ is a very general degree $10$ $K3$ surface. There exists a non--trivial birational involution
   \[ \iota\colon\ \ X\ \dashrightarrow\ X\ .\]
   There exists a $\iota$--invariant divisor $D\subset X$ (of Beauville--Bogomolov square $2$), such that the action of $\iota$ on the N\'eron--Severi group $NS(X)$ is given by reflection in the span of $D$.
   \end{proposition}
   
   \begin{proof} This is \cite[Section 4.3]{OG2} (cf. also \cite[Section 3.1]{IM}). The idea of the construction of $\iota$ is as follows. Using Mukai's work (theorem \ref{muk}), the $K3$ surface $S$ can be realized as a quadratic section $S=V_5\cap Q$ of the del Pezzo threefold $V_5:=G\cap \PP^6$. Hence, a general unordered pair of points on $S$ gives a general unordered pair of points $(x,y)$ on $V_5$. One checks (by a dimension count) that there is a unique conic $q=q_{x,y}\subset V_5$ passing through the pair of points $(x,y)$. Since 
   $S=V_5\cap Q$ is a quadratic section, the conic $q$ meets $S$ in $x,y$ plus $2$ other points $x^\prime, y^\prime$. The involution is defined by this residual intersection, i.e.
   \[ \iota(x,y):= (x^\prime,y^\prime)\ \ \ \in X\ .\]  
     \end{proof}
 
 \begin{theorem}[O'Grady \cite{OG}]\label{og2} Let $X$ and $\iota\in\bir(X)$ be as in proposition \ref{og1}. There exists a 
 hyperk\"ahler fourfold $X^\prime$ birational to $X$, and a generically $2:1$ morphism $p\colon X^\prime\to Y$ to an EPW sextic $Y$.
 
 Moreover, let $\iota^\prime\in\bir(X^\prime)$ be the birational involution induced by $\iota$. Then $Y\subset\PP^5$ is the closure of the quotient $U^\prime/\iota^\prime$, where $U^\prime\subset X^\prime$ is the open on which the involution $\iota^\prime$ is defined.
 \end{theorem}
 
 \begin{proof} This is contained in \cite[Theorem 4.15]{OG3}. The idea is that there is a generically $2:1$ rational map $X\dashrightarrow Y_A$ to an EPW sextic with $A\in\Delta$ in the notation of loc. cit. For $S$ very general, the subspace $A$ will be generic in $\Delta$ and thus $Y_A[3]$ will consist of a single point $v_0$. Let $X_A\to Y_A$ be the singular double cover of the EPW sextic as in loc. cit. According to \cite[Theorem 4.15]{OG3}, $X_A$ has one singular point $p_0$ (lying over $v_0\in Y_A$), and there exists a small resolution $s\colon X_A^\epsilon\to X_A$ with exceptional locus $E:=s^{-1}(p_0)$ isomorphic to $\PP^2$, and such that $X_A^\epsilon$ is isomorphic to the Hilbert square of a certain $K3$ surface (the $K3$ surface denoted $S_A(v_0)$ in loc. cit.). We define $X^\prime:=X_A^\epsilon$ and $Y:=Y_A$. 
 
 The singular variety $X_A$ has an involution $\iota_A\in\aut(X_A)$ (coinciding with $\iota\in\bir(X)$ on an open) such that 
 $Y=X_A/\iota_A$. Since $X_A^\epsilon\to X_A$ is birational, this proves the ``moreover'' statement. 
 \end{proof}

\section{Main result}

 \begin{theorem}\label{main} Let $X$ be the Hilbert scheme $S^{[2]}$, where $S$ is a very general $K3$ surface of degree $d=10$. Let $\iota\in\bir(X)$ be the anti--symplectic involution of proposition \ref{og1}. Then
       \[    \begin{split}  \iota^\ast=\ide\colon&\ \ \ A^4_{(0)}(X)\ \to\ A^4_{(0)}(X)\ ;\\                                  
                         \iota^\ast =-\ide\colon&\ \ \ A^4_{(2)}(X)\ \to\ A^4_{(2)}(X)\ ;\\
                             (\Pi_2^X)_\ast \iota^\ast=-\ide \colon& \ \ \ A^2_{(2)}(X)\ \to\ A^2_{(2)}(X)\ ;\\           
                                         (\Pi_4^X)_\ast      \iota^\ast =\ide\colon&\ \ \ A^4_{(4)}(X)\ \to\ A^4_{(4)}(X)\ .\\
                          \end{split}\]
  
   \end{theorem}   
   
\begin{proof} We first prove the statement for $A^2_{(2)}(X)$. The statements for $A^4_{(2)}(X)$ and for $A^4_{(4)}(X)$ will be deduced from the statement for $A^2_{(2)}(X)$ using theorem \ref{mult}.

% (As will be apparent to the well--informed reader, the proof of theorem \ref{main2} is directly inspired by Voisin's seminal work on the Bloch/Hodge equivalence for 
% complete intersections \cite{V0}, \cite{V1}, \cite{Vo}.) 
We consider the universal family
  \[ \Ss\ \to\ B \]
  of all smooth degree $10$ $K3$ surfaces $S_b$. Here the base $B$ is a Zariski--open in a product of projective spaces 
   \[   B\subset \bar{B}:=\PP H^0\bigl(G,\OO(1))  \bigr){}^{\times 3}\times \PP H^0\bigl( G,\OO(3)) \bigr)   \ ,\]
   corresponding to theorem \ref{muk}.
   
  We will write
  \[ \XX\ \to\ B \]
  for the universal family of Hilbert squares of degree $10$ $K3$ surfaces, and $X_b$ for a fibre of $\XX\to B$ over $b\in B$. This family is obtained from the family $\Ss\times_B \Ss$ (whose fibres are products $S_b\times S_b$) by a ``hat'' of morphisms over $B$
    \begin{equation}\label{hat} \begin{array}[c]{ccccc}
        && \wt{\Ss\times_B \Ss}&&\\
        &\swarrow&&\ \ \ \searrow&\\
        \XX &&&& \Ss\times_B \Ss\\
        \end{array}\end{equation}
        where $\wt{\Ss\times_B \Ss}$ is the blow--up of $\Ss\times_B \Ss$ with centre the relative diagonal, and the southwest arrow is the quotient morphism for the natural action of the symmetric group on $2$ elements. This diagram (\ref{hat}) gives rise to relative correspondences
        \[ \Psi\in A^4(\XX\times_B \Ss\times_B \Ss)\ ,\ \ \   {}^t \Psi\in A^4( \Ss\times_B \Ss\times_B \XX)  \ .\]
        (For details on relative correspondences, cf. \cite{MNP}, and also \cite{DM}, \cite{CH}, \cite{NS}.) Restricting to a fibre over $b\in B$, diagram (\ref{hat}) induces the familiar diagram
        \[ 
        \begin{array}[c]{ccccc}
        && \wt{S_b\times S_b}&&\\
        &\swarrow&&\ \ \ \searrow&\\
        X_b=(S_b)^{[2]} &&  && S_b\times S_b\\
        \end{array}    \]    
        (where $\wt{S_b\times S_b}$ is the blow--up of $S_b\times S_b$ along the diagonal),
    and the (absolute) correspondences 
    \[ \Psi_b\in A^4(X_b\times S_b\times S_b)\ ,\ \ {}^t \Psi_b\in A^4(S_b\times S_b\times X_b)\ .\]
      
 Since the construction of the birational involution $\iota_b\in\bir(X_b)$ of proposition \ref{og1} is geometric in nature, it naturally extends to the relative setting. More precisely,
 let 
   \[ \VV\ \to\ B^\prime\ \to\ B \]
   denote the family of smooth codimension $3$ linear sections of the Grassmannian $G=G(2,5)$ of lines in $\PP^4$ 
   (so $B^\prime$ is an open in $(\PP H^0(G,\OO_G(1)))^{\times 3}$, and each fibre $V_b$ of the family $\VV\to B$ is the del Pezzo threefold usually denoted $V_5$). Let $\FF\to B$ denote the family of Fano varieties of conics contained in $V_b$ (so the family $\FF\to B$ is isotrivial with fibre $F(V_5)$ according to the previous parenthesis).
   Associating to a general unordered pair of $2$ points on $S_b$ the unique conic in $V_b$ containing this pair of points defines a rational map of $B$--schemes
   \[  \XX\ \dashrightarrow\ \FF\ .\]
  Taking the residual intersection of the conic with the surface $S_b$, we get a birational involution
   of $B$--schemes
  \[ \iota\colon\ \ \XX\ \to\ \XX\ ,\]
  such that restriction to a fibre gives the birational involution $\iota_b\colon X_b\to X_b$ of proposition \ref{og1}.
  
  Let $\Gamma_\iota\in A^4(\XX\times_B \XX)$ denote the closure of the graph of the birational map $\iota$. The fact that $\iota_b$ acts as $-1$ on $H^{2,0}(X_b)$ for all $b\in B$, combined with the fact that $H^2_{tr}(X_b)\subset H^2(X_b)$ is the smallest Hodge substructure containing $H^{2,0}$, implies that
  \[ ({}^t \Gamma_{\iota_b}+\Delta_{X_b})\circ (\pi_{2,tr}^{X_b}) =0\ \ \ \hbox{in}\ H^8(X_b\times X_b)\ ,\ \ \ \forall b\in B\ .\]
 In view of the refined Chow--K\"unneth decomposition (theorem \ref{pi_2}), this implies that
  \begin{equation}\label{rel} ({}^t \Gamma_{\iota_b}+\Delta_{X_b})\circ (\pi_{2}^{X_b}) =\gamma_b\ \ \ \hbox{in}\ H^8(X_b\times X_b)\ ,\ \ \ \forall b\in B\ ,\end{equation}
  where $\gamma_b$ is some cycle supported on $Y_b\times Y_b$, for $Y_b\subset X_b$ a divisor.
  
%  Taking the transpose (or simply considering the action of $\iota$ on $H^6(X_b)$), one likewise finds the relation
%  \begin{equation}\label{rel} (\Gamma_{\iota_b}+\Delta_{X_b})\circ (\pi_{6}^{X_b}) =\gamma_b\ \ \ \hbox{in}\ H^8(X_b\times X_b)\ ,\ \ \ \forall b\in B\ ,\end{equation}
%  where $\gamma_b$ is some cycle supported on $Y_b\times Y_b$, for $Y_b\subset X_b$ a divisor.
  
  Let $\{\Pi^\XX_j\}$ be a relative MCK decomposition as in proposition \ref{relmck}. The relation (\ref{rel}) implies the following: the relative correspondence
  \[  \Gamma_0:= ({}^t \Gamma_{\iota}+\Delta_{\XX})\circ \Pi^\XX_{2}  \ \ \in A^4(\XX\times_B \XX) \]
  has the property that for each $b\in B$, there exists a divisor $Y_b\subset X_b$ and a cycle $\gamma_b$ supported on $Y_b\times Y_b$ such that
   \[ (\Gamma_0)\vert_{X_b\times X_b}=\gamma_b\ \ \ \hbox{in}\ H^8(X_b\times X_b)\ .\]  
   
  At this point, we recall  Voisin's ``spread--out'' result:
  
  \begin{proposition}[Voisin \cite{V0}]\label{spread} Let $\XX\to B$ be a smooth projective morphism of relative dimension $n$. Let $\Gamma\in A^n(\XX\times_B \XX)$ be a cycle such that
    for all $b\in B$, there exists a closed algebraic subset $Y_b\subset X_b$ of codimension $c$, and a cycle $\gamma_b\in A_n(Y_b\times Y_b)$ such that
    \[ \Gamma\vert_{X_b\times X_b}=\gamma_b\ \ \ \hbox{in}\ H^{2n}(X_b\times X_b)\ .\]
    Then there exists a closed algebraic subset $\YY\subset\XX$ of codimension $c$, and a cycle $\gamma\in A_\ast(\YY\times_B \YY)$ such that
    \[   \Gamma\vert_{X_b\times X_b}=\gamma\vert_{X_b\times X_b}\ \ \ \hbox{in}\ H^{2n}(X_b\times X_b)\ \ \ \forall b\in B.\]
    \end{proposition}
    
    \begin{proof} This is a Hilbert schemes argument \cite[Proposition 3.7]{V0}.
    \end{proof}
  
 Applying proposition \ref{spread} to $\Gamma_0$,  
  it follows there exists a divisor $\YY\subset\XX$ and a cycle $\gamma\in A_\ast(\YY\times_B \YY)$ such that 
    \[    (\Gamma_0 -\gamma )\vert_{X_b\times X_b}=0\ \ \ \hbox{in}\ H^8(X_b\times X_b)\ , \ \ \ \forall b\in B\ .\]
  That is, the relative correspondence
   \[ \Gamma_1:= \Gamma_0-\gamma  \ \ \in A^4(\XX\times_B \XX) \]
has the property of being homologically trivial on every fibre:
 \[ (\Gamma_1)\vert_{X_b\times X_b}=0   \ \ \ \hbox{in}\ H^8(X_b\times X_b)\ ,\ \ \ \forall b\in B.\]  
 
 At this point, it is convenient to consider the family $\Ss\times_B \Ss$ (of products of surfaces $S_b\times S_b$), rather than the family $\XX$ (of Hilbert schemes $(S_b)^{[2]}$). That is,
 we consider the relative correspondence
 \[ \Gamma_2:=    \Psi\circ  \Gamma_1   \circ {}^t \Psi\ \ \ \in A^4(\Ss^{4/B})\ ,\]
 where 
  \[ \Ss^{4/B}:= \Ss\times_B \Ss\times_B \Ss\times_B \Ss\ .\]
  Since 
  \[  (\Gamma_2)\vert_{(S_b)^4} = (\Psi_b)\circ ((\Gamma_1)\vert_{X_b\times X_b})\circ {}^t \Psi_b\ \ \ \hbox{in}\ A^4((S_b)^4)\ \]
 (restriction and composition commute), the relative correspondence $\Gamma_2$ has the property of being homologically trivial on every fibre:
 \begin{equation}\label{homvan2} (\Gamma_2)\vert_{(S_b)^4}=0   \ \ \ \hbox{in}\ H^8((S_b)^4)\ ,\ \ \ \forall b\in B.\end{equation}  
 
 Let us now define four relative correspondences
   \[ \Gamma^{k,\ell}_3:= \Theta_k \circ \Gamma_2\circ  \Xi_\ell  \ \ \ \in A^2(\Ss^{2/B})\ ,\ \ \ k,\ell\in\{1,2\}\ ,\]
   where $\Xi_\ell, \Theta_k$ are as in proposition \ref{prod2}.
   
 It follows from (\ref{homvan2}) there is fibrewise homological vanishing
   \[ ( \Gamma^{k,\ell}_3)\vert_{S_b\times S_b}=0\ \ \ \hbox{in}\ H^4(S_b\times S_b)\ \ \ \forall b\in B\ (k,\ell\in\{1,2\})\ .\]
   
   Applying the Leray spectral sequence argument of \cite[Lemmas 3.11 and 3.12]{V0}, one finds that there exist 
    \[ \delta_{k,\ell}\in\ima\bigl( A^4(G\times G\times B)\to A^2(\Ss\times_B \Ss)\bigr)\ \ \ (k,\ell\in\{1,2\} )\]
    such that (after replacing $B$ by a smaller Zariski open subset) there is global homological vanishing
   \[ \Gamma^{k,\ell}_3+\delta_{k,\ell}\ \ \ \in A^2_{hom}(\Ss\times_B \Ss)\ \ \ (k,\ell\in\{1,2\} )\ .\]
   But then, in view of proposition \ref{vanish}, we have that
   \[ \Gamma^{k,\ell}_3+\delta_{k,\ell}=0\ \ \ \in A^2_{}(\Ss\times_B \Ss)\ \ \ (k,\ell\in\{1,2\})\ .\]
   Composing on both sides, this implies there are also rational equivalences
   \begin{equation}\label{chvan}     \Xi_k \circ \Gamma^{k,\ell}_3  \circ \Theta_\ell  + \delta^\prime_{k,\ell}=0\ \ \ \in A^4(\Ss^{4/B})\ \ \ 
     (k,\ell\in\{1,2\})\ ,         \end{equation}
   where we define $\delta_{k,\ell}^\prime:=\Xi_k\circ \delta_{k,\ell}\circ \Theta_\ell$.

We note that the action of the restricted correspondences $\delta_{k,\ell}\vert_{S_b\times S_b}$ on $A^i(S_b)$ factors over
$A^{i+4}(G)$. Since the Grassmannian $G$ has trivial Chow groups, this implies that
  \[ (\delta_{k,\ell}\vert_{S_b\times S_b})_\ast =0\colon\ \ \ A^i_{hom}(S_b)\ \to\ A^i_{hom}(S_b)\ \ \ \forall b\in B\ \ \ 
  (k,\ell\in\{1,2\} )\ .\]  
%    We note that (since $G\times G$ has trivial Chow groups), the relative correspondences $\delta_{k,\ell}$ are ``completely decomposed'' (i.e., they are supported on 
  %    \[ \bigcup_i \VV_i\times_B \WW_i\ ,\ \ \  \hbox{with}\ \codim(\VV_i)+\codim(\WW_i)=2\ ).\] 
 %     This implies that also the compositions $\delta_{k,\ell}^\prime$
%   are completely decomposed (i.e., they are supported on 
 %   \[ \bigcup_i \VV^\prime_i\times \WW^\prime_i\ , \ \ \ \hbox{with}  \ \VV^\prime_i, \WW^\prime_i\subset \Ss^{2/B}\ ,\ \codim(\VV^\prime_i)+\codim(\WW^\prime_i)=4)\ .\] 
 %   It follows that for $b\in B$ general, the restrictions $\delta_{k,\ell}^\prime\vert_{(S_b)^4}$ will again be completely decomposed, and so (for dimension reasons)
 As $\delta_{k,\ell}^\prime$ is composed with $\delta_{k,\ell}$, the same property holds for $\delta_{k,\ell}^\prime$:
   \[ (\delta_{k,\ell}^\prime\vert_{(S_b)^4})_\ast=0\colon\ \ \ A^i_{hom}(S_b\times S_b)\ \to\ A^i_{hom}(S_b\times S_b)\ \ \ \forall b\in B\ \forall i\ \ \ (k,\ell\in\{1,2\} )\ .\]
   
   Plugging this in the restriction of equality (\ref{chvan}) to a fibre, we see that 
   \[\begin{split} \bigl( (\Xi_k\circ \Gamma^{k,\ell}_3\circ \Theta_\ell)\vert_{(S_b)^4}\bigr){}_\ast   =0\colon\ \ \ 
   A^i_{hom}&(S_b\times S_b)\ \to\  A^i_{hom}(S_b\times S_b) \ ,\\
   &\hbox{for\ all\ i\ and\ for\ \ all\ $b\in B$}\ \ \ (k,\ell\in\{1,2\} )\ . 
   \end{split}\]
  In view of the definition of the $\Gamma^{k,\ell}_3$, this implies that
  \[  \begin{split}  &\Bigl( (\Xi_1\circ\Theta_1+\Xi_2\circ \Theta_2)\circ \Gamma_2\circ (\Xi_1\circ\Theta_1+\Xi_2\circ 
     \Theta_2)\vert_{(S_b)^4}\Bigr){}_\ast \\
      &=
     \sum_{k,\ell\in\{1,2\}}      \bigl( (\Xi_k\Theta_k\circ \Gamma_2\circ \Xi_\ell\circ\Theta_\ell)\vert_{(S_b)^4}\bigr){}_\ast\\   
     &=0\colon\ \ \ 
    A^i_{hom}(S_b\times S_b)\ \to\  A^i_{hom}(S_b\times S_b) \ ,\ \ \ \hbox{for\ all\ i\ and\ for\ all\ $b\in B$}\ .\\
   \end{split} \]  
   But 
   \[ \bigl((\Xi_1\circ\Theta_1+\Xi_2\circ 
     \Theta_2)\vert_{(S_b)^4}\bigr){}_\ast  =  (\Pi_2^{(S_b)^2})_\ast\colon\ \ \ A^2_{}(S_b\times S_b)\ \to\ A^2(S_b\times S_b) \]
     (proposition \ref{prod2}), and $A^2_{(2)}\subset A^2_{hom}$, and so this simplifies to
     \begin{equation}\label{rat3}
          (\Pi_2^{(S_b)^2})_\ast (\Gamma_2\vert_{(S_b)^4}){}_\ast =0\colon\ \ \ A^2_{(2)}(S_b\times S_b)\ \to\ A^2_{(2)}(S_b\times S_b) \ \hbox{for\ $b\in B$\ general}\ .\end{equation}

 To finish the proof of the $A^2_{(2)}(X)$ part of theorem \ref{main}, it remains to connect the action of (the restriction of) $\Gamma_2$ and the action of (the restriction of) the relative correspondence $\Gamma_0$ that we started out with. We make this connection in the next two lemmas:
 
 \begin{lemma}\label{30} Notation as above. There is equality
   \[  (\Gamma_2\vert_{(S_b)^4})_\ast  =  \bigl((\Psi\circ\Gamma_0\circ {}^t \Psi)\vert_{(S_b)^4}\bigr){}_\ast \colon\ \ \ A^2_{hom}(S_b\times S_b)\ \to\ A^2_{hom}(S_b\times S_b)\ \ \    \hbox{for\ $b\in B$\ general}\ . \]
    \end{lemma}
    
  \begin{proof} Unravelling the various definitions we made, we find
    \[ \begin{split} \Gamma_2  &=      \Psi\circ \Gamma_1\circ {}^t \Psi\circ \Pi^{\Ss^{2/B}}_2\\
                                              &=         \Psi\circ (\Gamma_0-\gamma)\circ {}^t \Psi\circ \Pi^{\Ss^{2/B}}_2\\   
                                              &=   \Psi\circ \Gamma_0\circ {}^t \Psi\circ \Pi^{\Ss^{2/B}}_2 -\gamma^\prime\ \ \ \ \hbox{in}\ A^4(\Ss^{4/B})\ , \\
                                         \end{split}\]
                                   where $\gamma^\prime:= \Psi\circ \gamma\circ {}^t \Psi\circ \Pi^{\Ss^{2/B}}_2$ is a completely decomposed cycle.
                  The restriction of $\gamma^\prime$ to a general fibre $(S_b)^4$ will be a completely decomposed cycle, and as such will not act on $A^\ast_{hom}(S_b\times S_b)$. This proves the lemma.          
           \end{proof}

 \begin{lemma}\label{31} Notation as above. There is equality
   \[ ({}^t \Psi_b)_\ast (\Psi_b)_\ast=2\ide\colon\ \ A^2_{hom}(X_b)\ \to\ A^2_{hom}(X_b)\ \ \ \forall b\in B\ .\]  
  \end{lemma}
  
  \begin{proof} This is noted in the proof of lemma \ref{compat}.
  %Note that (since $H^3(X_b)=0$) there is equality $A^2_{hom}(X_b)=A^2_{AJ}(X_b)$. 
  \end{proof} 
  
 Obviously, lemmas \ref{30} and \ref{31} suffice to prove the $A^2_{(2)}(X)$ part of theorem \ref{main}: indeed, the combination of lemma \ref{30} with (\ref{rat3}) implies that
  \[    (\Pi_2^{(S_b)^2})_\ast    \bigl( (\Psi\circ\Gamma_0\circ {}^t \Psi)\vert_{(S_b)^4}\bigr){}_\ast   =0\colon\ \ \ A^2_{(2)}(S_b\times S_b)\ \to\  A^2_{}(S_b\times S_b) \ \ \ \hbox{for\ $b\in B$\ general}\ . \]
  Composing with $ ({}^t \Psi_b)_\ast$ on the left and with $(\Psi_b)_\ast$ on the right, 
  we find that
  \[  ({}^t \Psi_b)_\ast    (\Pi_2^{(S_b)^2})_\ast    \bigl( (\Psi\circ\Gamma_0\circ {}^t \Psi)\vert_{(S_b)^4}\bigr){}_\ast   (\Psi_b)_\ast  =0\colon\ \ \ A^2_{(2)}(X_b)\ \to\  A^2_{}(X_b) \ \ \ \hbox{for\ $b\in B$\ general}\ . \]  
  Applying relation (\ref{thishere}) and lemma \ref{31}, it follows that
  \[      (\Pi_2^{X_b})_\ast   (\Gamma_0)_\ast=0\colon\ \ \ A^2_{(2)}(X_b)  \ \to\ A^2_{}(X_b)\ \ \ \hbox{for\ $b\in B$\ general}\ . \]  
  By virtue of the definition of $\Gamma_0$ (and the fact that $A^2_{(2)}(X_b)=(\Pi_2^{X_b})_\ast A^2_{hom}(X_b)$), it follows that
   \[   (\Pi_2^{X_b})_\ast    (\Gamma_{\iota_b} +\Delta_{X_b})_\ast=0\colon\ \ \ A^2_{(2)}(X_b)\ \to\ A^2_{(2)}(X_b)\ \ \ \hbox{for\ }b\in B\ \hbox{general}\ ,\]
   as asserted by theorem \ref{main}.

  We have now proven the $A^2_{(2)}(X)$ part of theorem \ref{main}. The statement for $A^4_{(4)}(X)$ follows easily from this. Indeed, 
  Shen--Vial have proven the multiplication map
  \[ A^2_{(2)}(X)\otimes A^2_{(2)}(X)\ \to\ A^4_{(4)}(X) \]
  is surjective (theorem \ref{mult}). Given $b\in A^4_{(4)}(X)$, we can thus write
    \[ b= a_1\cdot a_2\ \ \ \hbox{in}\ A^4(X)\ ,\]
    where $a_1, a_2\in A^2_{(2)}(X)$. The statement we have just proven for $A^2_{(2)}$ implies that
    \[  \iota^\ast(a_j)=-a_j + r_j\ \ \ \hbox{in}\ A^2(X)\ \ \ (j=1,2)\ ,\]
    where $r_j\in A^2_{(0)}(X)\cap A^2_{hom}(X)$. It follows that
     \[ \iota^\ast(b) = \iota^\ast(a_1\cdot a_2)=\iota^\ast(a_1)\cdot \iota^\ast(a_2)= (-a_1+r_1)\cdot (-a_2+r_1)=a_1\cdot a_2 -r_1\cdot a_2 -r_2\cdot a_1\ \ \ \hbox{in}\ A^4(X)\ .\]
    (NB: note that $\iota$ is not a morphism, and so the second equality is {\em not\/} trivial. The second equality happens to be true since $a_1, a_2\in A^2_{AJ}(X)$; this is \cite[Proposition B.6]{SV}.)
    But then, since $-r_1 \cdot a_2-r_2\cdot a_1\in A^4_{(2)}(X)$, we have
    \[ (\Pi_4^X)_\ast \iota^\ast(b) = a_1\cdot a_2=b\ \ \ \hbox{in}\ A^4_{(4)}(X) \]
    as requested.
    
    It remains to prove theorem \ref{main} for $A^4_{(0)}(X)$ and $A^4_{(2)}(X)$. As we have seen (theorem \ref{mult}), Shen--Vial have shown there exists a class $l\in A^2_{(0)}(X)$ inducing an isomorphism
    \begin{equation}\label{sviso}  \cdot l\colon\ \ A^2_{(2)}(X)\ \xrightarrow{\cong}\ A^4_{(2)}(X)\ .\end{equation}
    
    We need to understand the action of $\iota$ on the class $l\in A^2(X)$. To this end, we will prove the following:
    
       \begin{proposition}\label{onl} Let $X$ and $\iota$ be as in theorem \ref{main}. Let $l\in A^2_{(0)}(X)$ be the class as in theorem \ref{mult}.
    Then
        \[ \iota^\ast(l) =\pm l\ \ \ \hbox{in}\ A^2(X)\ .\]
        \end{proposition}  
    
    Proposition \ref{onl} suffices to complete the proof of theorem \ref{main}. Indeed, one has
      \[ A^4_{(0)}(X) =\QQ [l^2] \]
      \cite[Theorem 4.6]{SV}. The action of $\iota$ on $l^2$ satisfies
      \[ \iota^\ast(l^2) = \iota^\ast(l)\cdot \iota^\ast(l) + \iota^\ast\bigl( (l-\iota_\ast\iota^\ast(l))\cdot   (l-\iota_\ast\iota^\ast(l))\bigr) = l^2\ \ \ \hbox{in}\ A^4(X)\ .\]
      (Here, for the first equality we have used \cite[Lemma B.4]{SV}, and the second equality follows from proposition \ref{onl}.) This proves the $A^4_{(0)}$ part of theorem \ref{main}.
      
   It remains to prove the $A^4_{(2)}$ part of theorem \ref{main}.       
   For this, let us suppose for a moment that proposition \ref{onl} is true with a minus sign, i.e.
    \[ \iota^\ast(l)=-l\ \ \ \hbox{in}\ A^2(X)\ .\]
    Using the isomorphism (\ref{sviso}), \cite[Proposition B.6]{SV}, and the fact that (as proven above) $\iota^\ast=-\ide$ on $A^2_{(2)}(X)$, this would imply
    \begin{equation}\label{contra}  \iota^\ast=\ide\colon\ \ \ A^4_{(2)}(X)\ \to\ A^4(X)\ .\end{equation}
  The statement for $A^4_{(4)}(X)$ we have just proven is that for any $b\in A^4_{(4)}(X)$ we have
  \[ \iota^\ast(b) = b+r\ \ \ \hbox{in}\ A^4(X)\ \]
  where $r\in A^4_{(2)}(X)$. Since $\iota^\ast\iota^\ast(b)=b$, this implies that $\iota^\ast(r)=-r$, and so (using equality (\ref{contra})) $r=0$.
    That is, we find that
    \[  \iota^\ast=\ide\colon\ \ \ A^4_{}(X)\ \to\ A^4(X)\ .\]
    Applying \cite[Lemma 3.1]{SV} to ${}^t \bar{\Gamma_\iota}-\Delta_X$, this would imply that
    \[ {}^t \bar{\Gamma_\iota}-\Delta_X = \gamma\ \ \ \hbox{in}\ A^4(X\times X)\ ,\]
    where $\gamma$ is a cycle supported on $X\times D$ for $D\subset X$ a divisor. In particular, this would imply
    \[ \iota^\ast=\ide\colon\ \ \ H^{2,0}(X)\ \to\ H^{2,0}(X)\ ,\]
    which is absurd since we know that $\iota$ is non--symplectic. The minus sign in proposition \ref{onl} can thus be excluded; assuming proposition \ref{onl} is true,
    we must have $\iota^\ast(l)=l$.
      
   Now let $c\in A^4_{(2)}(X)$. Using the isomorphism (\ref{sviso}), we can find $a\in A^2_{(2)}(X)$ such that
    \[ c=l\cdot a\ \ \ \hbox{in}\ A^4(X)\ .\] 
    We know (from the $A^2_{(2)}(X)$ statement proven above) that $\iota^\ast(a)=-a+r$, where $r\in A^2_{(0)}(X)\cap A^2_{hom}(X)$.
    But then
    \[ \iota^\ast(c)=\iota^\ast(l\cdot a)=  \iota^\ast(l)\cdot \iota^\ast(a) = l\cdot (-a+r) = -l\cdot a =-c\ \ \ \hbox{in}\ A^4(X)\ .\]
    Here, the second equality holds thanks to \cite[Proposition B.6]{SV}, and the third equality comes from proposition \ref{onl} and the statement for $A^2_{(2)}$. The fourth equality uses $A^2_{(0)}(X)\cdot (A^2_{(0)}(X)\cap A^2_{hom}(X))=0$, which is a consequence of $A^4_{(0)}(X)\cap A^4_{hom}(X)=0$.
    This proves theorem \ref{main}, assuming proposition \ref{onl}.
        
   We now proceed with the proof of proposition \ref{onl}. The first step is to prove the corresponding statement in homology:
   
   \begin{lemma}\label{onlhom} Let $S$ be any $K3$ surface and let $X=S^{[2]}$. Let $l\in A^2(X)$ be the class of theorem \ref{mult}, and let $\iota\in\bir(X)$ be any birational involution. Then we have
     \[  \iota^\ast(l) =\pm l\ \ \ \hbox{in}\ H^4(X)\ .\]   
   \end{lemma}
   
 \begin{proof} Shen and Vial have constructed a distinguished cycle $L\in A^2(X\times X)$ (whose cohomology class is the Beauville--Bogomolov class denoted $ \BB$ in loc. cit.), and an eigenspace decomposition
   \begin{equation}\label{eigen}  A^2(X)=\Lambda^2_{25}\oplus \Lambda^2_2\oplus \Lambda^2_0\ ,\end{equation}
   where 
   \[ \Lambda^i_\lambda := \{ a\in A^i(X)\ \vert\ (L^2)_\ast(a)=\lambda a\}\ ,\]
   and
   \[ \Lambda^2_{25}=\QQ[l]\ \]
 (This is \cite[Theorem 14.5, Propositions 14.6 and 14.8]{SV}, combined with \cite[Theorem 2.2]{SV}).
   
   We now observe the following commutativity relation in cohomology:
   
   \begin{lemma}\label{Li}  Set--up as in lemma \ref{onlhom}. Then
   \[ (L^2)_\ast\iota^\ast = \iota^\ast (L^2)_\ast\colon\ \ \ H^i(X)\ \to\ H^i(X)\ .\]
   \end{lemma}
   
 \begin{proof} Let $L\in A^2(X\times X)$ be the Shen--Vial cycle as above. As proven in \cite[Proposition 1.3(\rom1)]{SV}, the Shen--Vial cycle satisfies a quadratic relation
   \begin{equation}\label{quadr}  L^2 = 2\Delta_X -{2\over 25} (l_1+l_2) L - {1\over 23\cdot 25} (2l_1^2 -23l_1l_2 +2l_2^2)\ \ \ \hbox{in}\ H^8(X\times X)\ ,
          \end{equation}  
 where $l:=(i_\Delta)^\ast(L)$ (and $i_\Delta\colon X\to X\times X$ is the diagonal embedding) and $l_i:= (p_i)^\ast(l)$ (and $p_i$ are the obvious projections).
 
 Let us define a modified cycle
   \[ L^\prime:= \Gamma_\iota\circ L\circ \Gamma_\iota\ \ \ \in A^2(X\times X)\ .\]
   Using Lieberman's lemma (\cite[16.1.1]{F} or \cite[Lemma 3.3]{V3}), plus the fact that ${}^t \Gamma_\iota=\Gamma_\iota$, we see that
   \[ L^\prime= (\iota\times\iota)^\ast (L)\ \ \ \hbox{in}\  A^2(X\times X)\ .\]   
   Define also $l^\prime:=(i_\Delta)^\ast(L^\prime)\in A^2(X)$ and $l^\prime_i:=(p_i)^\ast(l^\prime)\in A^2(X\times X)$, $i=1,2$.
   Since the diagram
   \[ \begin{array}[c]{ccccc}  X\times X & \xrightarrow{p_i} & X & \xrightarrow{ i_\Delta} & X\times X\\
                             \ \ \ \ \ \ \  \ \  \downarrow{\iota\times\iota} && \ \ \ \ \ \downarrow{\iota} &&\ \ \ \ \ \ \ \ \  \ \downarrow{\iota\times\iota}\\
                                                X\times X & \xrightarrow{p_i} & X & \xrightarrow{ i_\Delta} & X\times X\\
                                                \end{array}\]
                                   commutes, we have the relations
                          \begin{equation}\label{ll}  l^\prime_i = (\iota\times\iota)^\ast (l_i)\ \ \ \hbox{in}\ A^2(X\times X)\ ,\ \ \ i=1,2\ .\end{equation}
       Let us apply $(\iota\times\iota)^\ast$ to the quadratic relation (\ref{quadr}). The result is a relation
       \begin{equation}\label{step} \begin{split} (\iota\times\iota)^\ast(L^2)=2\Delta_X -{2\over 25}(\iota\times\iota)^\ast (l_1+l_2) L^\prime - {1\over 23\cdot 25} (\iota\times\iota)^\ast(&2l_1^2 -23l_1l_2 +2l_2^2)\\   &\ \ \ \hbox{in}\ H^8(X\times X)\ .\\ \end{split}\end{equation}
       But 
       \begin{equation}\label{sq} (\iota\times\iota)^\ast(L^2) = \bigl((\iota\times\iota)^\ast L\bigr){}^2= (L^\prime)^2\ \ \ \hbox{in}\ A^4(X\times X)\ .\end{equation}
       Plugging this in equality (\ref{step}), and also using the relations (\ref{ll}), we find that the cycle $L^\prime$ satisfies a quadratic relation
 \begin{equation}\label{quadrp}  (L^\prime)^2 = 2\Delta_X -{2\over 25} (l^\prime_1+l^\prime_2) L^\prime - {1\over 23\cdot 25} (2(l^\prime_1)^2 -23l^\prime_1 l^\prime_2 +2(l^\prime_2)^2)\ \ \ \hbox{in}\ H^8(X\times X)\ .
          \end{equation}  
        
        But then, applying the unicity result \cite[Proposition 1.3 (\rom5)]{SV}, we find there is equality  
        \[ L^\prime=\pm L\ \ \ \hbox{in}\ H^4(X\times X)\ .\]
        In particular, there is equality
        \[  (L^\prime)^2= L^2\ \ \ \hbox{in}\ H^8(X\times X)\ .\]
        In view of equality (\ref{sq}), this means
        \[  \Gamma_\iota\circ (L^2)\circ \Gamma_\iota =L^2\ \ \ \hbox{in}\ H^8(X\times X)\ ,\]               
         and so (by composing with $\Gamma_\iota$)
        \[  \Gamma_\iota\circ (L^2) = (L^2)\circ \Gamma_\iota \ \ \ \hbox{in}\ H^4(X\times X)\ .\]  
        This proves lemma \ref{Li}.

  \end{proof}
    
  The eigenspace decomposition (\ref{eigen}) induces an eigenspace decomposition modulo homological equivalence:
    \[  \ima \bigl(   A^2(X)\ \to\ H^4(X) \bigr) =  \Lambda^2_{25} +  {\Lambda^2_2 \over A^2_{(0)}(X)\cap A^2_{hom}(X) }   \]
    (this is the algebraic part of the eigenspace decomposition of $H^4(X)$ given in \cite[Proposition 1.3(\rom3)]{SV}).
   
   Lemma (\ref{Li}) implies $\iota$ preserves this eigenspace decomposition modulo homological equivalence. In particular, $\iota^\ast \Lambda^2_{25}\subset\Lambda^2_{25}$ (modulo homologically trivial cycles), and so
   \[  \iota^\ast(l) = d l\ \ \ \hbox{in}\ H^4(X)\ ,\]
   for some $d\in\QQ$.
   Since $\iota$ is an involution, we must have $d=\pm 1$. This proves lemma \ref{onlhom}.
   \end{proof}  
     
  The next step (in proving proposition \ref{onl}) is to upgrade to rational equivalence. Here, we use again the method of ``spread'' developed in \cite{V0}, \cite{V1}. As above, let $\Ss\to B$ resp. $\XX\to B$ denote the family of all smooth degree $10$ $K3$ surfaces $S_b\subset\PP^3$, resp. of all Hilbert schemes $X_b=(S_b)^{[2]}$. 
  We note that there exists a relative cycle
   \[ \LL\in A^2(\XX) \]
   such that restriction
   \begin{equation}\label{fib} \LL\vert_{X_b} =l_b\ \ \ \in A^2(X_b)\ \ \ \forall b\in B\end{equation} 
   is the distinguished class (denoted $l$ in theorem \ref{mult}) for the fibre $X_b$. Indeed, one can define $\LL$ as
   \[ \LL:={5\over 6} c_2(T_{\XX/B})\ \ \ \in A^2(\XX)\ ,\]
   where $T_{\XX/B}$ is the relative tangent bundle of the smooth morphism $\XX\to B$. Since for any $b\in B$ there is a relation
    \[ l_b={5\over 6}c_2(X_b)\ \ \ \hbox{in}\ A^2(X_b)\] 
    \cite[Equation (93)]{SV}, this implies (\ref{fib}).
  
  The relative cycle
   \[ \Gamma_0:=  \LL \pm \iota^\ast (\LL)\ \ \ \in A^2(\XX)\]
   is such that the restriction to each fibre is homologically trivial:
   \[ (\Gamma_0)\vert_{X_b}=0\ \ \ \hbox{in}\ H^4(X_b)\ .\]
   (Here, ``$\pm$'' is taken to mean $+$ (resp. $-$) if lemma \ref{onlhom} is true with a $-$ (resp. a $+$).)
  Thus, the relative cycle 
   \[ \Gamma_1:=  \Psi_\ast (\Gamma_0)\ \ \ \in A^2(\Ss\times_B \Ss) \]
   also is homologically trivial on each fibre. (Here, $\Psi$ is the relative correspondence from $\XX$ to $\Ss\times_B \Ss$ as in the proof of theorem \ref{main}.)
   
  Applying \cite[Lemma 3.12]{V0}, up to shrinking $B$ we can make $\Gamma_1$ globally homologically trivial. That is, there exists 
    \[ \psi\ \ \in \ima \bigl( A^2(B\times \PP^3\times \PP^3)\ \to\ A^2(\Ss\times_B \Ss)\bigr)\]
     such that
  (after replacing $B$ by a non--empty open subset $B^\prime\subset B$)
   \[ \Gamma_2 := \Gamma_1 +\psi\ \ \ \in A^2(\Ss\times_{B^\prime} \Ss) \]
   is actually in  $A^2_{hom}(\Ss\times_{B^\prime} \Ss)$.
   
   But  $A^2_{hom}(\Ss\times_{B^\prime} \Ss)=0$ (proposition \ref{vanish}), and so
   \[ \Gamma_2=0\ \ \ \hbox{in}\ A^2 (\Ss\times_{B^\prime} \Ss)\ . \]
Restricting to a fibre, we find 
 \[  (\Gamma_1)\vert_{S_b\times S_b} + \psi\vert_{S_b\times S_b}  =0\ \ \ \hbox{in}\ A^2(S_b\times S_b)\ \ \ \forall b\in B^\prime\ .\]
 As $\Gamma_1$ is fibrewise homologically trivial, the same goes for $\psi$:
   \begin{equation}\label{vanpsi}   \psi\vert_{S_b\times S_b}  =0\ \ \ \hbox{in}\ H^4(S_b\times S_b)\ \ \ \forall b\in B^\prime\ .\end{equation}
   But $A^2(\PP^3\times\PP^3)=\oplus_i A^i(\PP^3)\otimes A^{2-i}(\PP^3)$ and so
   \[  \psi\vert_{S_b\times S_b}=  \lambda_0 [S_b]\times H_b^2 + \lambda_1 H_b\times H_b + \lambda_2 H_b^2\times [S_b]\ \ \ \hbox{in}\ A^2(S_b\times S_b)\ ,\]
   where $\lambda_i\in\QQ$ and $H_b\in A^1(S_b)$ is an ample class on $S_b$. It follows from the vanishing (\ref{vanpsi}) that the $\lambda_i$ must be $0$, and so 
    $ \psi\vert_{S_b\times S_b}$ is rationally trivial, and hence also
    \[   (\Gamma_1)\vert_{S_b\times S_b}   =  0\ \ \ \hbox{in}\ A^2(S_b\times S_b)\ .\]  
    Composing with ${}^t \Psi_b$, it follows that also
     \[ ({}^t \Psi_b)_\ast \bigl((\Gamma_1)\vert_{S_b\times S_b}\bigr)=({}^t \Psi_b)_\ast (\Psi_b)_\ast \bigl((\Gamma_0)\vert_{X_b}\bigr)=0 \ \ \ \hbox{in}\ A^2(X_b)\ \ \ \forall b\in B^\prime\ .\]
     On the other hand, as we have seen above $(\Gamma_0)\vert_{X_b}\in A^2_{hom}(X_b)$ and   $   ({}^t \Psi_b)_\ast (\Psi_b)_\ast$ is the identity on $A^2_{hom}(X_b)$.
     It follows that
     \[      (\Gamma_0)\vert_{X_b}= \bigl(l_b \pm (\iota_b)^\ast(l_b)\bigr)\vert_{X_b}   =0\ \ \ \hbox{in}\ A^2(X_b)\ \ \ \forall b\in B^\prime\ .\]
     This proves proposition \ref{onl} for general $b\in B$. To extend to {\em all\/} $b\in B$, one can invoke \cite[Lemma 3.2]{Vo}. Proposition \ref{onl} and hence theorem \ref{main} are now proven.
    
 \end{proof}   
   
 %  For later use, we remark that the above argument also proves the following statement:

    \begin{remark}\label{rem} Can one prove the commutativity of lemma \ref{Li} also modulo rational equivalence, i.e. can one prove
   \begin{equation}\label{Lirat} (L^2)_\ast\iota^\ast \stackrel{??}{=} \iota^\ast (L^2)_\ast\colon\ \ \ A^i(X)\ \to\ A^i(X)\ \ ? \end{equation}
 This would imply that $\iota$ respects the eigenspace decomposition $\Lambda^i_\lambda$ of \cite{SV} (and in particular, that $\iota$ respects the bigraded ring structure $A^\ast_{(\ast)}(X)$).
 
 The proof of lemma \ref{Li} given above does not extend to rational equivalence, for the following reason:
 The quadratic relation (\ref{quadr}) still holds modulo rational equivalence \cite[Theorem 14.5]{SV}, and so $L^\prime$ satisfies the quadratic relation (\ref{quadrp}) modulo rational equivalence.
 However, the unicity result (\cite[Proposition 1.3(\rom5)]{SV}), that allowed us to conclude from this that $L=\pm L^\prime$, is only known modulo homological equivalence.  
 
 (This unicity result modulo rational equivalence is conjecturally true, and would follow from the Bloch--Beilinson conjectures \cite[Proposition 3.4]{SV}.)
  \end{remark}

 \section{Some consequences}

 \subsection{A birational statement} Theorem \ref{main} can be extended to other birational models:
 
 \begin{corollary}\label{birat} Let $X$ and $\iota$ be as in theorem \ref{main}. Let $X^\prime$ be a hyperk\"ahler variety birational to $X$, and let $\iota^\prime\in\bir(X^\prime)$ be the birational involution induced by $\iota$. Then
   \[    \begin{split} ( \iota^\prime)^\ast=\ide\colon&\ \ \ A^4_{(0)}(X^\prime)\ \to\ A^4_{(0)}(X^\prime)\ ;\\                                  
                         (\iota^\prime)^\ast =-\ide\colon&\ \ \ A^4_{(2)}(X^\prime)\ \to\ A^4_{(2)}(X^\prime)\ ;\\
                             (\Pi_2^{X^\prime})_\ast (\iota^\prime)^\ast=-\ide \colon& \ \ \ A^2_{(2)}(X^\prime)\ \to\ A^2_{(2)}(X^\prime)\ ;\\           
                                         (\Pi_4^{X^\prime})_\ast      (\iota^\prime)^\ast =\ide\colon&\ \ \ A^4_{(4)}(X^\prime)\ \to\ A^4_{(4)}(X^\prime)\ .\\
                          \end{split}\] 
  \end{corollary}

\begin{proof} First, let us recall (lemma \ref{hk}) that $X^\prime$ has an MCK decomposition (induced by an MCK decomposition for $X$, and the birational map 
$\phi\colon X\dashrightarrow X^\prime$ induces isomorphisms
  \[ \phi_\ast\colon\ \ \ A^i_{(j)}(X)\ \xrightarrow{\cong}\ A^i_{(j)}(X^\prime)\ .\]
%  (this follows from Rie\ss's result \cite{Rie} that $\phi^\ast\phi_\ast$ and $\phi_\ast\phi^\ast$ are the identity).
  
  To deduce corollary \ref{birat} from theorem \ref{main}, it only remains to establish commutativity of the diagram
    \begin{equation}\label{square} \begin{array}[c]{ccc}
                 A^i_{(j)}(X) & \xrightarrow{\phi_\ast}& A^i_{(j)}(X^\prime)\\
                \ \ \ \  \downarrow{\scriptstyle \iota^\ast} &&   \ \ \ \  \downarrow{\scriptstyle (\iota^\prime)^\ast}\\
                A^i_{}(X) & \xrightarrow{\phi_\ast}& A^i_{}(X^\prime)\\    
              \end{array}\end{equation}
        for the relevant $(i,j)$. Let $U\subset X$, $U^\prime\subset X^\prime$ be opens such that $\iota$ is everywhere defined on $U$ and $\phi$ induces an isomorphism between $U$ and $U^\prime$. Any $0$--cycle $a\in A^4(X)$ is represented by a cycle $\alpha$ with support contained in $U$. Then $\phi_\ast(a)$ is represented by the cycle with isomorphic support in $U^\prime$. This proves commutativity of the square (\ref{square}) for $i=4$. The Bloch--Srinivas argument \cite{BS} (or, more precisely, \cite[Lemma 3.1]{SV}) applied to the correspondence
      \[ \bar{\Gamma}_\phi\circ \bar{\Gamma}_\iota   - \bar{\Gamma}_{\iota^\prime}\circ\bar{\Gamma}_\phi\ \ \ \in A^4(X\times X) \]
   (where $\bar{\Gamma}$ indicates closure of the cycle $\Gamma\subset X\times X$) then shows commutativity for $A^2_{AJ}(X)=A^2_{hom}(X)$, and so (since $A^2_{(2)}\subset A^2_{hom}$) the square (\ref{square}) commutes for $(i,j)=(2,2)$.   
        %This can be done as follows: consider a ``hat'' of morphisms
       % \[ \begin{array}  [c]{ccccc}
         %    && Z &&\\
     %      & {\scriptstyle p\ }\swarrow\ \  &&\ \  \searrow{\scriptstyle\ q}&\\
    %   \ \ \ \  X & &\stackrel{\phi}{\dashrightarrow} & &X^\prime\ \ \ \ \\
    %    \end{array}\]
    %    resolving the indeterminacy of $\phi$. We may suppose $\phi$ extends to an involution $\iota_Z\in\aut(Z)$. There is a diagram
    %    \[ \begin{array}[c]{ccccc}
      %   A^i_{(j)}(X) &\xrightarrow{p^\ast}& A^i(Z) &\xrightarrow{q_\ast}& A^i_{(j)}(X^\prime)\\
      %   \downarrow{\scriptstyle \iota^\ast} && \downarrow{\scriptstyle (\iota_Z)^\ast} &&  \downarrow{\scriptstyle (\iota^\prime)^\ast}\\  
       %   A^i_{}(X) &\xrightarrow{p^\ast}& A^i(Z) &\xrightarrow{q_\ast}& A^i_{}(X^\prime)\\
       %   \end{array}\]
       \end{proof}

 \subsection{EPW sextics}   
 
Let $X$ and $\iota$ be as in theorem \ref{main}. As we have seen, there is a birational modification $X^\prime$ of $X$ that is a hyperk\"ahler fourfold, and such that there is a generically $2:1$ morphism from $X^\prime$ to an EPW sextic $Y$ (\cite{OG}, cf. theorem \ref{og2} above). 
%Since $Y$ is a quotient variety, the Chow groups with $\QQ$--coefficients form a ring. NOOOOOO ! there is one point that is NOT a quotient sing. !!!!!!!!
The following result is about the Chow ring $A^\ast$ (in the sense of operational Chow cohomology \cite{F}) of the EPW sextic $Y$. We note that for any variety $M$, there exists a ``cycle class'' map 
  \[A^i(M)\to \gr^W_{2i} H^{2i}(M)\] 
  which is functorial, and agrees with the usual cycle class map for smooth $M$ \cite{Tot}.

\begin{corollary}\label{epw}  Let $X$ be as in theorem \ref{main}, and let 
  $Y\subset\PP^5$ be the associated EPW sextic.
  For any $r\in\NN$, let
  \[ E^\ast(Y^r)\ \subset\ A^\ast(Y^r) \]
  be the subring generated by (pullbacks of) $A^1(Y)$ and $A^2(Y)$. The cycle class map 
  \[   E^k(Y^r)\ \to\ \gr^W_{2k}H^{2k}(Y^r) \]
  is injective for $k\ge 4r-1$.
  \end{corollary}
  
  \begin{proof} Let $X^\prime$ and $\iota^\prime$ be as in theorem \ref{og2}. The point is that $X^\prime$, and hence also $(X^\prime)^r$, has an MCK decomposition \cite{SV}. Let $p\colon X^\prime\to Y$ denote the morphism of theorem \ref{og2}. 
  
  \begin{lemma}\label{in0} We have
    \[ p^\ast A^2(Y)\ \subset\ A^2_{(0)}(X^\prime)\ .\]
   \end{lemma}
   
   \begin{proof} As explained in the proof of theorem \ref{og2}, the morphism $p$ decomposes as
   \[ p\colon\ \ \ X^\prime:=X_A^\epsilon\ \xrightarrow{s}\ X_A\ \xrightarrow{q}\ Y\ ,\]
   where $s$ is a small resolution of the singular variety $X_A$, and $q$ is a double cover with covering involution $\iota_A$
   (and $\iota_A$ agrees with $\iota^\prime$ on the open where $\iota^\prime$ is defined). Because of the equality $q\circ \iota_A=q\colon X_A\to Y$, one has an inclusion
   \[ q^\ast A^i(Y)\ \subset\ A^i(X_A)^{\iota_A}\ \ \ \forall i\ .\]
   Because of the equality $\iota_A\circ s = s\circ \iota^\prime\colon X^\prime\to X_A$, one has an inclusion
   \[  s^\ast \bigl ( A^i(X_A)^{\iota_A} \bigr)\ \subset\ A^i(X^\prime)^{\iota^\prime}\ \ \ \forall i\ .\]
     
   Combining these two inclusions and taking $i=2$, we find in particular that
    \[ p^\ast A^2(Y)\ \subset\ A^2(X^\prime)^{\iota^\prime}\ .\]

     Given $b\in A^2(Y)$, let us write
     \[ p^\ast(b) = c_0 + c_2\ \ \ \in A^2_{(0)}(X^\prime)\oplus A^2_{(2)}(X^\prime)\ .\]
     Applying $\iota^\prime$, we find
       \begin{equation}\label{this}   (\iota^\prime)^\ast p^\ast(b) = p^\ast(b)=c_0 + c_2 \ \ \ \in A^2_{(0)}(X^\prime)\oplus A^2_{(2)}(X^\prime)\ .\end{equation}
      On the other hand, corollary \ref{birat} implies that
      \[  (\iota^\prime)^\ast(c_2)=-c_2+r\ \ \ \hbox{in}\ A^2(X^\prime)\ ,\]
      where $r\in A^2_{(0)}(X^\prime)$.
      It follows that      
            \begin{equation}\label{thistoo} (\iota^\prime)^\ast p^\ast(b) = (\iota^\prime)^\ast(c_0) +(\iota^\prime)^\ast(c_2) =(\iota^\prime)^\ast(c_0)+r  - c_2\ \ \ \in A^2_{(0)}(X^\prime)\oplus A^2_{(2)}(X^\prime)\ \end{equation}
     (where we have used lemma \ref{preserve} below to obtain that $(\iota^\prime)^\ast(c_0)\in A^2_{(0)}(X^\prime)$).  
     Comparing expressions (\ref{this}) and (\ref{thistoo}), we find
      \[  (\iota^\prime)^\ast(c_0)+r = c_0\ \ \ \hbox{in}\ A^2_{(0)}(X^\prime)\ ,\ \ \ -c_2=c_2\ \ \ \hbox{in}\ A^2_{(2)}(X^\prime)\ ,\]
      proving lemma \ref{in0}.

      \begin{lemma}\label{preserve} Set--up as above. 
      Let $b\in A^2(Y)$, and write
       \[ p^\ast(b) = c_0 + c_2\ \ \ \in A^2_{(0)}(X^\prime)\oplus A^2_{(2)}(X^\prime)\ .\]   
       Then
            \[ (\iota^\prime)^\ast (c_0)\in\  A^2_{(0)}(X^\prime)\ .\]
  \end{lemma}
  
  \begin{proof} 
   Suppose
   \[ (\iota^\prime)^\ast(c_0) = d_0 + d_2\ \ \ \hbox{in}\ A^2(X^\prime)\ ,\]  
   with $d_0\in A^2_{(0)}(X^\prime)$ and $d_2\in A^2_{(2)}(X^\prime)$.
 %  Applying theorem \ref{main2}, we find that
  % \[ b = \iota^\ast(c_0 + c_2) = \iota^\ast(c_0) -c_2\ \ \ \hbox{in}\ A^2(X)\ .\]
   
   Let $\gamma\in A^4(X\times X^\prime)$ be the correspondence of \cite{Rie} (cf. also lemma \ref{hk}) inducing an isomorphism of bigraded rings
   \begin{equation}\label{isoul} \gamma_\ast\colon\ \ \ A^\ast_{(\ast)}(X)\ \xrightarrow{\cong}\ A^\ast_{(\ast)}(X^\prime)\ .\end{equation}
   
   Let $l\in A^2_{(0)}(X)$ be the distinguished class of theorem \ref{mult}, and define
       \[l^\prime:= \gamma_\ast(l)\ \ \   \in \ A^2_{(0)}(X^\prime)\ .\]
   The $0$--cycle $c_0\cdot l^\prime$ is in $A^4_{(0)}(X^\prime)$, and so corollary \ref{birat} implies that
   \begin{equation}\label{here} (\iota^\prime)^\ast (c_0\cdot l^\prime) = c_0\cdot l^\prime\ \ \ \hbox{in}\ A^4_{(0)}(X^\prime)\ .\end{equation}
        
   On the other hand, we have
   \begin{equation}\label{andhere}   (\iota^\prime)^\ast (c_0\cdot l^\prime)  =  (\iota^\prime)^\ast(c_0)\cdot (\iota^\prime)^\ast(l^\prime)=  (d_0+ d_2)\cdot l^\prime = d_0\cdot l^\prime + d_2\cdot l^\prime \ \ \ \hbox{in}\ A^4(X^\prime) \ .\end{equation}
   (Here in the first equality, we have used sublemma \ref{comiota} below, and in the second equality we have used proposition \ref{onl}, which we have seen must be true with a $+$ sign.) 
      
   Since $d_0\cdot l^\prime\in A^4_{(0)}(X^\prime)$ and $d_2\cdot l^\prime\in A^4_{(2)}(X^\prime)$, comparing expressions (\ref{andhere}) and (\ref{here}), we see that we must have
     \[  d_0\cdot l^\prime = c_0\cdot l^\prime\ \ \ \hbox{in}\ A^4_{(0)}(X^\prime)\ ,\ \ \ d_2\cdot l^\prime=0\ \ \ \hbox{in}\ A^4_{(2)}
      (X^\prime)\ .\]
     In view of the isomorphism (\ref{isoul}) and the injectivity part of theorem \ref{mult}, this implies that $d_2=0$, thus proving lemma \ref{preserve}.
      \end{proof}  
    
   Above, we have used the following sublemma:
    
  \begin{sublemma}\label{comiota} Let $c_0\in A^2_{(0)}(X^\prime)$ be as above. For any $e\in A^2(X^\prime)$, 
   there is equality
   \[ (\iota^\prime)^\ast (c_0\cdot e )  =  (\iota^\prime)^\ast (c_0)\cdot  (\iota^\prime)^\ast (e)  \ . \]
   \end{sublemma}
   
   \begin{proof} As we have seen above, the morphism $p$ is a composition of morphisms 
      \[p\colon X^\prime\ \xrightarrow{s}\ X_A\ \xrightarrow{}\ Y\ ,\]
      where $X_A$ is the ``singular double EPW sextic'' on which there exists an involution $\iota_A\in\aut(X_A)$ extending $\iota$. We thus have
      \[ p^\ast A^2(Y)\ \subset\ s^\ast A^2(X_A)\ \]
      (where we recall that $A^\ast()$ of the singular varieties $Y$ and $X_A$ means the operational Chow cohomology of \cite{F}). 
     Let $b_A :=q^\ast(b)\in A^2(X_A)$, so that $p^\ast(b)=s^\ast(b_A)$ in $A^2(X^\prime)$.
     We claim that there is equality
     \begin{equation}\label{onXpr}  (\iota^\prime)^\ast (s^\ast(b_A)\cdot e )  =  (\iota^\prime)^\ast s^\ast(b_A)\cdot  (\iota^\prime)^\ast (e) \ \ \ \forall e\in A^2(X^\prime) \ . \end{equation}  
     Since there is also equality
      \[  (\iota^\prime)^\ast (c_2\cdot e )  =  (\iota^\prime)^\ast (c_2)\cdot  (\iota^\prime)^\ast (e) \ \ \ \forall e\in A^2(X^\prime) \ \] 
      (because $c_2\in A^2_{(2)}(X^\prime)\subset A^2_{AJ}(X^\prime)$, this follows from \cite[Proposition B.6]{SV}), and the definitions imply that $c_0=s^\ast(b_A)-c_2$, the claim suffices to prove the sublemma.   
      
     To prove the claim, we exploit the fact that there is a commutative square
    \[  \begin{array}[c]{ccc}
      X^\prime &\xdashrightarrow{\iota^\prime}& X^\prime\\
      \downarrow{\scriptstyle s}&& \downarrow{\scriptstyle s}\\
      X_A &\xrightarrow{\iota_A}& X_A\\
      \end{array}\]
      
  Since $s_\ast\colon A_0(X^\prime)\to A_0(X_A)$ is an isomorphism, it suffices to prove equality (\ref{onXpr}) holds after pushing forward under $s$. The pushforward of the left--hand side of (\ref{onXpr}) is
  \[ \begin{split} s_\ast (\iota^\prime)^\ast (s^\ast(b_A)\cdot e )      &=  ((\iota_A)^{-1})_\ast s_\ast (s^\ast (b_A)\cdot e)\\
                     &=  (\iota_A)^\ast s_\ast (s^\ast (b_A)\cdot e)\\                     
                       &= (\iota_A)^\ast (b_A\cap s_\ast(e))\\
                       %&=  (\iota_A)_\ast \bigl( (\iota_A)^\ast (\iota_A)^\ast(b_A)\cap s_\ast(e)\bigr)\\
                       &= (\iota_A)^\ast (b_A)\cap (\iota_A)_\ast s_\ast(e)\ \ \ \hbox{in}\ A_0(X_A)\ .\\
                       \end{split}\]
%          (Here, we define $(\iota_A)^\ast := ((\iota_A)^{-1})_\ast    \colon A_\ast(X_A)\to A_\ast(X_A)$ and so we have $(\iota_A)^\ast     =(\iota_A)_\ast$.)
(Here the notation $ \alpha\cap\beta$ means the action of an operational Chow cohomology class $\alpha\in A^\ast(X_A)$ on $\beta\in A_\ast(X_A)$. The third equality is an application of the projection formula as given in \cite[Definition 17.3]{F}.)

          The pushforward of the right--hand side of (\ref{onXpr}) is
        \[ \begin{split}     s_\ast \bigl( (\iota^\prime)^\ast s^\ast(b_A)\cdot  (\iota^\prime)^\ast (e) \bigr) &=
                           s_\ast \bigl( s^\ast (\iota_A)^\ast (b_A)   \cdot (\iota^\prime)^\ast(e)\bigr)    \\
                           &= (\iota_A)^\ast (b_A)\cap s_\ast (\iota^\prime)^\ast(e) \\
                           & = (\iota_A)^\ast (b_A)\cap  ((\iota_A)^{-1})_\ast s_\ast(e) \\
                           & = (\iota_A)^\ast (b_A)\cap (\iota_A)_\ast s_\ast(e)\ \ \ \hbox{in}\  A_0(X_A)\ .\\
                           \end{split}\]
                           This proves the claim, and hence sublemma \ref{comiota}.
      \end{proof}

     \end{proof}
  
  Lemma \ref{in0}, combined with the obvious fact that $A^1(X^\prime)=A^1_{(0)}(X^\prime)$, implies that
    \[  (p^r)^\ast E^\ast(Y^r)\ \subset\ A^\ast_{(0)}((X^\prime)^r)\ .\]
  Since there is a commutative diagram
    \[ \begin{array}[c]{ccc}   
             A^k_{(0)}((X^\prime)^r) & \to & H^{2k}((X^\prime)^r)  \\
           \ \ \ \ \ \   \uparrow{\scriptstyle (p^r)^\ast} &&    \ \ \ \ \ \   \uparrow{\scriptstyle (p^r)^\ast}\\
           E^k(Y^r) & \to & \ \ \gr^W_{2k} H^{2k}(Y^r)\ , \\
           \end{array}\]
               and the cycle class map
    \[  A^k_{(0)}((X^\prime)^r)\ \to\ H^{2k}((X^\prime)^r) \]
    is known to be injective for $k\ge 4r-1$ (\cite[Introduction]{V6}; this follows for instance from \cite[Section 4.3]{V4}), this establishes corollary \ref{epw}.
    \end{proof}

  We single out a particular case of corollary \ref{epw}:
  
  \begin{corollary}\label{epw1}  Let $Y\subset\PP^5$ be an EPW sextic as in corollary \ref{epw}. The subspaces
        \[   \begin{split} &\ima \Bigl(  A^2(Y)\otimes A^1(Y)\ \to\ A^3(Y)\Bigr)\ ,\\
                             &\ima \Bigl(  A^2(Y)\otimes A^2(Y)\ \to\ A^4(Y)\Bigr)\ \\
                             \end{split}\]
  are of dimension $1$. 
  \end{corollary}
  
  \begin{proof} This follows from corollary \ref{epw}, combined with the fact that
    \[ N^3(Y):= \ima \Bigl( A^3(Y)\ \to\ \gr^W_6 H^6(Y)\Bigr) \]
    is of dimension $1$. To see this, since the pairing
    \[ NS(X)^\iota\otimes N^3(X)^\iota\ \to\ N^4(X)^\iota\cong\QQ \]
    is non--degenerate, it suffices to prove 
    \begin{equation}\label{N1}  \dim NS(Y) =\dim NS(X)^\iota=1\ .\end{equation}
    But $\dim \gr^W_{2}H^2(Y)=1$ (weak Lefschetz for the hypersurface $Y\subset\PP^5$), and so
    \[ \dim H^2(X)^\iota=1\ ,\]
    proving (\ref{N1}).
    
    (Alternatively, (\ref{N1}) can also be proven directly: $\iota$ acts on $NS(X)$ as reflection in the span of $D$ (proposition \ref{og1}), and so $NS(X)^\iota=\QQ[D]$ is of dimension $1$.)
     \end{proof}

%  \begin{remark} For {\em any\/} EPW sextic $Y$, conjecturally the cycle class map
%   \[   E^k(Y^r)\ \to\ \gr^W_{2k}H^{2k}(Y^r) \]
%  is injective for all $k$ and all $r$ (this is an analogue of the Beauville--Voisin conjecture for powers of $K3$ surfaces \cite{BV}, cf. \cite{EPWmoi}). This is true if $Y$ has an MCK decomposition and, in addition, the conjectural vanishing
%    \[  A^2_{(0)}(Y)\cap A^2_{hom}(Y)\stackrel{??}{=}0 \]
%    holds.  
%    
%    Also, the pullbacks of the diagonal $\Delta_Y$ should be in the subring $E^\ast(Y^r)$. (I have not been able to prove corollary \ref{epw} for this larger subring
%    containing pullbacks of $\Delta_Y$. The problem is that one needs to prove that $\Gamma_\iota$ is in $A^4_{(0)}(X\times X)$, i.e. that $\iota$ is ``of pure grade $0$'' in the sense of \cite[Definition 1.1]{SV2}.)   
%  \end{remark}
  
  \begin{remark}\label{compare} It is instructive to compare corollary \ref{epw1} with known results concerning the Chow ring of $K3$ surfaces and of Calabi--Yau varieties. For any $K3$ surface $S$, it is known that
   \[ \dim \ima \Bigl( A^1(S)\otimes A^1(S)\ \to\ A^2(S)\Bigr) =1 \]
   \cite{BV}.
   For a generic Calabi--Yau complete intersection $X$ of dimension $n$, it is known that
   \[ \dim \ima \Bigl( A^i(X)\otimes A^{n-i}(X)\ \to\ A^n(X)\Bigr)=1\ ,\ \ \ \forall 0<i<n \]
   \cite{V13}, \cite{LFu}.
  
  The new part of corollary \ref{epw1}, with respect to these results, is the part about 
   \[\ima\Bigl(A^2(X)\otimes A^1(X)\to A^3(X)\Bigr)\ .\] 
   This part is conjecturally true for all EPW sextics \cite{EPWmoi}, but presumably {\em not\/} true for general Calabi--Yau varieties (or even general Calabi--Yau complete intersections). This is related to the question of determining which varieties satisfy Beauville's weak splitting property \cite{Beau3}.
  \end{remark}

 \section{Some questions}
 \label{finsec}

In this final section, we record some questions suggested by the results of this note.

\begin{question} Let $X$ and $\iota$ be as in theorem \ref{main}. Can one prove that $\iota$ preserves the bigrading $A^\ast_{(\ast)}(X)$ of the Chow ring (so that the $\Pi_j^X$ in theorem \ref{main} can be omitted) ? It is a matter of regret that I have not been able to prove this (cf. remark \ref{rem}).
\end{question}

\begin{question} Let $X$ and $\iota$ be as in theorem \ref{main}. What can one say about the action of $\iota$ on $A^3(X)$ ?
This seems more difficult than theorem \ref{main}. Indeed, the action of $\iota$ on $A^2_{hom}$ and on $A^4$ is determined by ``behaviour up to codimension $1$ phenomena''. The action of $\iota$ on $A^3_{(2)}$, on the other hand, should be determined by the action of $\iota$ on $H^{3,1}(X)$, which is not as neat as the action on $H^{2,0}(X)$ and $H^{4,0}(X)$. I am not even sure what the conjectural statement should be.
\end{question}

%\begin{question} Let $X$ and $\iota$ be as in theorem \ref{main}. Does the EPW sextic $Y=X/\iota$ have a (self--dual) MCK decomposition ? I have not been able to prove this (essentially, this reduces to the problem of showing that $\iota$ is of pure grade $0$, in the sense of \cite[Definition 1.1]{SV2}).
%\end{question} 

\begin{question} Let $X^\prime$ be a double EPW sextic as in theorem \ref{og2}, and let $E\cong\PP^2$ denote the exceptional locus of the small resolution $X^\prime\to X_A$. 
Since $E\subset X^\prime$ is a constant cycle subvariety, the ideas developed in \cite{V14} suggest that $E$ should lie in $A^2_{(0)}(X^\prime)$. Can this actually be proven ?
\end{question}

\begin{question}\label{iliev-madonna} Let $X$ be a Hilbert square $X=S^{[2]}$, where $S$ is a very general $K3$ surface of degree 
  \[d=2(4n^2+8n+5)\ ,\ \ \ n\in\NN\ .\] 
  Generalizing O'Grady's result (theorem \ref{og2}, which corresponds to $n=0$), Iliev--Madonna show that $X$ is birational to a double EPW sextic \cite[Theorem 1.1]{IM}, and so there exists an anti--symplectic birational involution $\iota$ on $X$. It follows from Shen--Vial's work (theorem \ref{mck}) that $X$ has an MCK decomposition. Is it possible to prove theorem \ref{main} for $X$ 
and $\iota$ ? 

The case $n=1$ (i.e. $d=34$ and $g(S)=18$) can perhaps be done similarly to theorem \ref{main}: in this case, the general $K3$ surface $S$ is given as zero locus of some vector bundle $V$ on the orthogonal Grassmannian $\og(3,9)$ \cite{Muk1}. The only problem is that I'm not sure whether the vector bundle $V$ is sufficiently ample for Voisin's method of spread to work in this case.

The case $n>1$ seems considerably more difficult, due to the absence of Mukai models for $S$.
\end{question}

\begin{question} Double EPW sextics form a $20$--dimensional family of hyperk\"ahler fourfolds of $K3^{[2]}$ type (and the above--mentioned work \cite{IM} constructs a countably infinite number of codimension $1$ subfamilies, elements of which are birational to Hilbert squares of $K3$ surfaces). More ambitiously than the above question \ref{iliev-madonna}, can one somehow prove theorem \ref{main} for {\em all\/} double EPW sextics ?

One obvious problem is that first (in order for the statement of theorem \ref{main} even to make sense), one would need to construct an MCK decomposition for a general double EPW sextic.
\end{question}

\vskip1cm
\begin{nonumberingt} Thanks to all participants of the Strasbourg 2014/2015 ``groupe de travail'' based on the monograph \cite{Vo} for a very pleasant atmosphere.
Much thanks to Jiji and Baba for hospitably receiving me in Kokuryo, where this note was written \smiley.
\end{nonumberingt}

\end{document}